\documentclass{amsart}%
\usepackage{amsmath}
\usepackage{amsfonts}
\usepackage{amssymb}
\usepackage{graphicx}%
\newtheorem{lemma}{Lemma}[section]
\newtheorem{theorem}[lemma]{Theorem}
\newtheorem{remark}[lemma]{Remark}

\newtheorem{coro}[lemma]{Corollary}
\newtheorem{definition}[lemma]{Definition}
\newtheorem{example}[lemma]{Example}

\title{The Integration of Stepanov Remotely Almost Periodic Functions}

\author{David Cheban}
\address[D. Cheban]{State University of Moldova\\
Vladimir Andrunachievici Institute of Mathematics and Computer Science\\
Laboratory of Differential Equations\
str. Academiei 5\\
MD--2028 Chi\c{s}in\u{a}u, Moldova} \email[D.
Cheban]{david.ceban@usm.md, davidcheban@yahoo.com}

\date{\today}

\begin{document}

\begin{abstract}
The aim of this paper is to study the problem of the integration
of Stepanov remotely almost periodic functions. We prove that
every compact primitive of a Stepanov remotely almost periodic
function with a minimal $\omega$-limit set is remotely almost
periodic. This fact proves the conjecture previously formulated by
the author.
\end{abstract}

\maketitle

\section{Introduction}\label{Sec1}

In this article, we continue the research started in the author's
works \cite{Che_2024.2,Che_2024.3} in which remotely almost
periodic functions are studied. Let $\mathbb
R:=(-\infty,+\infty),\ \mathbb R_{+}:=[0,+\infty)$, $\mathbb T\in
\{\mathbb R_{+},\mathbb R\}$, $\mathfrak B$ be a Banach space
over field $P$ ($P=\mathbb R$ or the set of all complex numbers
$\mathbb C$) and $C(\mathbb T,\mathfrak B)$ be the space of all
continuous mappings $\varphi :\mathbb T\to \mathfrak B$ equipped
with the compact open topology. Denote by $(C(\mathbb T,\mathfrak
B),\mathbb T,\sigma)$ the shift dynamical system on $C(\mathbb
T,\mathfrak B)$, where $\sigma$ is the mapping from $\mathbb T
\times C(\mathbb T,\mathfrak B)$ to $C(\mathbb T,\mathfrak B)$
defined by $\sigma(h,\varphi):=\varphi^{h}$ and
$\varphi^{h}(t):=\varphi(t+h)$ for all $(t,\varphi)\in \mathbb T
\times C(\mathbb T,\mathfrak B)$.

A subset $A$ from $\mathbb T$ is called relatively dense in
$\mathbb T$ if there exists a positive number $l$ such that
$[a,a+l]\bigcap A\not= \emptyset$ for all $a\in \mathbb T$, where
$[a,a+l]:=\{t\in \mathbb T|\ a\le t\le a+l\}$.

Recall \cite{Che_2024.2} that a function $\varphi \in C(\mathbb
T,\mathfrak B)$ is said to be:
\begin{enumerate}
\item almost periodic if for every positive number $\varepsilon$ the
set
\begin{equation}\label{eqIAP1}
\mathcal P(\varepsilon,\varphi):=\{\tau\in \mathbb T|\
|\varphi(t+\tau)-\varphi(t)|<\varepsilon\ \ \mbox{for all}\ \ t\in
\mathbb T\}\nonumber
\end{equation}
is relatively dense in $\mathbb T$; \item asymptotically almost
periodic if there are two functions $p,r\in C(\mathbb T,\mathfrak
B)$ such that the following two conditions are fulfilled:
\begin{enumerate}
\item $\varphi(t)=p(t)+r(t)$ for all $t\in \mathbb T$; \item the
function $p$ is almost periodic and $r\in C_{0}(\mathbb
T,\mathfrak B)$, where by $C_{0}(\mathbb T,\mathfrak B)$ the
family of all functions from $C(\mathbb T,\mathfrak B)$ vanishing
at the $+\infty$, i.e., $|r(t)|\to 0$ as $t\to +\infty$;
\end{enumerate}
\item remotely almost periodic if for every $\varepsilon >0$ there
exists a relatively dense in $\mathbb T$ subset $\mathcal
P(\varepsilon,\varphi)$ such that for every $\tau \in \mathcal
P(\varepsilon,\varphi)$ we have a positive number
$L(\varepsilon,\varphi,\tau)$ so that
\begin{equation}\label{eqIAP2}
|\varphi(t+\tau)-\varphi(t)|<\varepsilon \nonumber
\end{equation}
for all $|t|\ge L(\varepsilon,\varphi,\tau)$.
\end{enumerate}

The notion of remotely almost periodicity (on the real axis
$\mathbb R:=(-\infty,+\infty)$) for scalar function was introduced
and studied by Sarason \cite{Sar_1984}. Recall that a continuous
and bounded function $f:\mathbb R \to \mathbb R$ is said to be
remotely almost periodic if for every $\varepsilon >0$ there
exists a positive number $l$ such that on every segment
$[a,a+l]\subset \mathbb R$ there exists at least one number $\tau
\in [a,a+l]$ such that
\begin{equation}\label{eqI1}
    d_{\infty}(f^{\tau},f)<\varepsilon ,\ \ \mbox{where}\ \ d_{\infty}(f,g):=\limsup\limits_{|t|\to
    +\infty}|f(t)-g(t)| \nonumber
\end{equation}
and $f^{\tau}(t):=f(t+\tau)$ for all $t\in \mathbb R$. The
remotely almost periodic ($RAP$) functions form a closed
subalgebra, $RAP$, of $BUC$ (the algebra of bounded and uniformly
continuous complex valued functions on $\mathbb R$). The main
result of Sarason is that $RAP$ is generated as a Banach algebra,
by $AP$ (the algebra of Bohr almost periodic functions) and
another algebra $SO$ (consisting of functions which oscillate
slowly at $\infty$).

Remotely almost periodic functions on the semi-axis $\mathbb
R_{+}:=[0,+\infty)$ with values in the Banach space were
introduced by Ruess and Summers \cite{RS_1986} (see also Baskakov
\cite{Bas_2013} and \cite{Che_2024.0}-\cite{Che_2024.2}).

Let $\varphi \in C(\mathbb T,\mathfrak B)$ be a remotely almost
periodic function and
\begin{equation}\label{eqPQ1}
\Phi (t)=\int_{0}^{t}\varphi(s)ds \nonumber
\end{equation}
be a compact primitive of the function $\varphi$, i.e.,
$Q:=\overline{\varphi(\mathbb T)}$ is a compact subset of $\mathbb
T$.

If the Banach space $\mathfrak B$ is one-dimensional, then the
compact primitive $\Phi$ is also remotely almost periodic
\cite{ZP_2011}.

Let $c_0$ be the Banach space of all numerical sequences
$\{\xi_k\}_{k=1}^{\infty}$ that converges to $0$ with the norm
$\|\xi\|:=\sup\{|\xi_{k}|\ |\ k=1,2,\ldots\}$.

We will say that a Banach space $\mathfrak B$ does not contain
$c_0$ if it has no subspace isomorphic to $c_0$.

In the work \cite{VMM_2022} was established the following result:
the bounded primitive $\Phi$ of asymptotically almost periodic
function $\varphi$ is remotely almost periodic if and only if the
Banach space $\mathfrak B$ does not contain $c_0$.

In author's paper \cite{Che_2024.2} the remotely almost
periodicity of the compact primitive $\Phi$ of asymptotically
almost periodic function $\varphi$ was established.

The following conjecture was formulated in the work
\cite{Che_2024.2}.

\textbf{Conjecture}. Let $\varphi \in C(\mathbb T,\mathfrak B)$ be
a remotely almost periodic (respectively, remotely $\tau$-periodic
or remotely stationary) functions. Assume that the following
conditions are fulfilled:
\begin{enumerate}
\item the function $\varphi$ is positively Lagrange stable, i.e.,
the set $\Sigma_{\varphi}^{+}:=\{\varphi^{h}|\ h\ge 0\}$ is
precompact in $C(\mathbb T,\mathfrak B)$; \item the $\omega$-limit
set $\omega_{\varphi}$ of the function $\varphi$ is a minimal set
of the shift dynamical system $(C(\mathbb T,\mathfrak B),\mathbb
T,\sigma)$.
\end{enumerate}

Then the compact primitive $\Phi$ of the function $\varphi$ is
remotely almost periodic (respectively, remotely $\tau$-periodic
or remotely stationary).

The main result of this paper gives a positive answer to above
conjecture.

In this paper we study the problem of integration of Stepanov
remotely almost periodic functions. This study continues the
author's series of works dedicated to the study of remotely almost
periodic motions of dynamical systems and solutions of
differential equations
\cite{Che_2009},\cite{Che_2024.0}-\cite{Che_2024.1} and
\cite{Che_2024.4,Che_2023wp}.

The paper is organized as follows. In the second Section, we
collect some known notions of dynamical systems and facts about
remotely almost periodic motions and remotely almost periodic
functions that we use in this paper. The third Section contains
the main results of the paper and it is dedicated to the study the
problem of integration of Stepanov remotely almost periodic
functions. We prove that every compact primitive of the Stepanov
remotely almost periodic function $\varphi$ with the minimal
$\omega$-limit set $\omega_{\varphi}$ is remotely almost periodic.

\section{Preliminary}\label{Sec2}

\subsection{Remotely almost periodic motions of dynamical
systems}\label{Sec2.1}

Let $(X,\rho_{X})$ and $(Y,\rho_{Y})$ be two complete metric
spaces with the distance $\rho_{X}$ and $\rho_{Y}$
respectively\footnote{In what follows, in the notation $\rho_{X}$
(respectively, $\rho_{Y}$), we will omit the index $X$
(respectively, $Y$) if this does not lead to a misunderstanding.},
let $\mathbb R:=(-\infty,+\infty)$, $\mathbb Z :=\{0,\pm 1, \pm 2,
\ldots \}$, $\mathbb S =\mathbb R$ or $\mathbb Z$,  $\mathbb
S_{+}=\{t \in \mathbb S |\quad t \ge 0 \}$ and $\mathbb S_{-}=\{t
\in \mathbb S| \quad t \le 0 \}$. Let $\mathbb T \in \{\mathbb S,\
\mathbb S_{+}\}$ and $(X,\mathbb S_{+},\pi)$ (respectively,
$(Y,\mathbb S, \sigma )$) be an autonomous one-sided
(respectively, two-sided) dynamical system on $X$ (respectively,
$Y$).

Let $(X,\mathbb T,\pi)$ be a dynamical system.

\begin{definition}\label{defSP1} A point $x\in X$ (respectively, a motion $\pi(t,x)$) is
said to be:
\begin{enumerate}
\item[-] stationary, if $\pi(t,x)=x$ for every $t\in \mathbb T$;
\item[-] $\tau$-periodic ($\tau >0$ and $\tau \in \mathbb T$), if
$\pi(\tau,x)=x$; \item[-] asymptotically stationary (respectively,
asymptotically $\tau$-periodic), if there exists a stationary
(respectively, $\tau$-periodic) point $p\in X$ such that
\begin{equation}\label{eqAP1*}
\lim\limits_{t\to \infty}\rho(\pi(t,x),\pi(t,p))=0.\nonumber
\end{equation}
\end{enumerate}
\end{definition}

\begin{theorem}\label{thAAP1}\cite[Ch.I]{Che_2009} A point $x\in X$ is asymptotically $\tau$-periodic if and
only if the sequences $\{\pi(k\tau,x)\}_{k=0}^{\infty}$ converges.
\end{theorem}

\begin{definition}\label{defLS1} A point $\widetilde{x}\in X$ is said
to be $\omega$-limit (respectively, $\alpha$-limit) for $x\in X$
if there exists a sequence $\{t_k\}\subset \mathbb S_{+}$
(respectively, $\{t_k\}\subset \mathbb S_{-}$) such that $t_k\to
+\infty$ (respectively, $-\infty$) and $\pi(t_k,x)\to
\widetilde{x}$ as $k\to \infty$.
\end{definition}

Denote by $\omega_{x}$ (respectively, $\alpha_{x}$) the set of all
$\omega$-limit (respectively, $\alpha$-limit) points of $x\in X$.

\begin{definition}\label{defSAP1}  We will call
a point $x\in X$ (respectively, a motion $\pi(t,x)$) remotely
$\tau$-periodic ($\tau\in \mathbb T$ and $\tau
>0\widetilde{}$) if
\begin{equation}\label{eqSAP_1}
 \lim\limits_{t\to
+\infty}\rho(\pi(t+\tau,x),\pi(t,x))=0 .
\end{equation}
\end{definition}

\begin{remark}\label{remS1.0} The motions of dynamical systems possessing the property
(\ref{eqSAP_1}) was studied in the works of Cryszka
\cite{Gry_2018} and Pelczar \cite{Pel_1985}.
\end{remark}

\begin{definition}\label{defLS2} A point $x$ is called positively Lagrange
stable, if the semi-trajectory $\Sigma_{x}^{+}:=\{\pi(t,x)|\ t\in
\mathbb S_{+}\}$ is a precompact subset of $X$.
\end{definition}

\begin{theorem}\label{th1.3.9}\cite[Ch.I]{Che_2020} Let $x\in X$ be
positively Lagrange stable and $\tau\in\mathbb T\ (\tau >0)$. Then the following
statements are equivalent:
\begin{enumerate}
\item[a.] the motion $\pi(t,x)$ is remotely $\tau$-periodic;
\item[b.] every point $p\in\omega_{x}$ is $\tau$-periodic.
\end{enumerate}
\end{theorem}

\begin{definition}\label{defSAP2} A point $x$ (respectively, a
motion $\pi(t,x)$) is said to be remotely stationary, if it is
remotely $\tau$-periodic for every $\tau \in \mathbb T$.
\end{definition}

\begin{coro}\label{corSAP1} Let $x\in X$ be positively Lagrange stable. Then the following
statements are equivalent:
\begin{enumerate}
\item[a.] the motion $\pi(t,x)$ is remotely stationary; \item[b.]
every point $p\in\omega_{x}$ is stationary.
\end{enumerate}
\end{coro}
\begin{proof} This statement follows directly from the
corresponding definition and Theorem \ref{th1.3.9}.
\end{proof}

\begin{definition}\label{defAP1} A point $x\in X$ of dynamical
system $(X,\mathbb T,\pi)$ is said to be:
\begin{enumerate}
\item almost recurrent if for every $\varepsilon >0$
\begin{equation}\label{eqAR1}
\mathcal T(\varepsilon,p):=\{\tau \in \mathbb T|\
\rho(\pi(\tau,p),p)<\varepsilon\}\nonumber
\end{equation}
is relatively dense in $\mathbb T$; \item recurrent if it is
almost recurrent and Lagrange stable;
 \item almost periodic if for every
$\varepsilon
>0$ the set
\begin{equation}\label{eqAP1}
\mathcal P(\varepsilon,p):=\{\tau \in \mathbb T|\
\rho(\pi(t+\tau,p),\pi(t,p))<\varepsilon \ \ \mbox{for all}\ t\in
\mathbb T\}\nonumber
\end{equation}
is relatively dense in $\mathbb T$; \item positively
(respectively, negatively) Poisson stable if $x\in \omega_{x}$
(respectively, $x\in \alpha_{x}$); \item asymptotically stationary
(respectively, asymptotically $\tau$-periodic or asymptotically
almost periodic) if there exists a stationary (respectively,
$\tau$-periodic or almost periodic) point $p\in X$ such that
\begin{equation}\label{eqAP3}
\lim\limits_{t\to \infty}\rho(\pi(t,x),\pi(t,p))=0.\nonumber
\end{equation}
\end{enumerate}
\end{definition}

\begin{definition}\label{defIS1} A subset $M\subseteq X$ is said
to be positively invariant (respectively, negatively invariant or
invariant) if $\pi(t,M)\subseteq M$ (respectively, $M\subseteq
\pi(t,M)$ or $\pi(t,M)=M$) for all $t\in \mathbb T$.
\end{definition}

\begin{definition}\label{defRAP2} A subset $M$ is said to be
equi-almost periodic if for every $\varepsilon >0$ there exists a
relatively dense subset $\mathcal P(\varepsilon,M)$ such that
\begin{equation}\label{eqRAP2}
\rho(\pi(t+\tau,p),\pi(t,p))<\varepsilon \nonumber
\end{equation}
for all $t\in \mathbb T$, $\tau \in \mathcal P(\varepsilon,M)$ and
$p\in M$.
\end{definition}

\begin{lemma}\label{lRAP1}\cite{RS_1986} Let $x\in X$ be a
positively Lagrange stable point of the dynamical system
$(X,\mathbb S_{+},\pi)$. The motion $\pi(t,x)$ is remotely almost
periodic if and only if its $\omega$-limit set $\omega_{x}$ is
equi-almost periodic.
\end{lemma}

Consider a two-sided dynamical system $(X,\mathbb S,\pi)$.

\begin{definition}\label{defTSR1} A point $x\in X$ (respectively, a motion $\pi(t,x)$) is said to be
two-sided remotely almost periodic if the following two conditions
are fulfilled:
\begin{enumerate}
\item the point $x$ is Lagrange stable, i.e., the set
$\Sigma_{x}:=\{\pi(t,x)|\ t\in \mathbb S\}$ is precompact; \item
for every $\varepsilon >0$ there exists a relatively dense in
$\mathbb S$ subset $\mathcal P(\varepsilon,x)$ such that
\begin{equation}\label{eqTSR1}
\lim\sup\limits_{|t|\to
+\infty}\rho(\pi(t+\tau,x),\pi(t,x))<\varepsilon \nonumber
\end{equation}
or equivalently for every $\varepsilon >0$ and $\tau \in \mathcal
P(\varepsilon,x)$ there exists a positive number
$L(\varepsilon,x,\tau)$ such that
\begin{equation}\label{eqTSE2}
\rho(\pi(t+\tau,x),\pi(t,x))<\varepsilon \nonumber
\end{equation}
for all $|t|\ge L(\varepsilon,x,\tau)$.
\end{enumerate}
\end{definition}

\begin{lemma}\label{lRAP_010} \cite{Che_2024.1} A point $x$
is remotely $\tau$-periodic (respectively, remotely stationary) if
and only if for every $\varepsilon
>0$ there exists a relatively dense in $\mathbb T$ subset $\mathcal
P(x,\varepsilon)$ such that
\begin{enumerate}
\item for every $\tau \in \mathcal P(x,\varepsilon)$ there exists a
number $L(x,\varepsilon,\tau)>0$ for which we have
\begin{equation}\label{eqE1}
\rho(\pi(t+\tau,x),\pi(t,x))<\varepsilon \nonumber
\end{equation}
for all $t\ge L(x,\varepsilon,\tau)$ and \item $\{\tau \mathbb Z\}
\subset \mathcal P(x,\varepsilon)$ (respectively, $\mathbb T
\subseteq \mathcal P(x,\varepsilon)$).
\end{enumerate}
\end{lemma}

\begin{theorem}\label{thTSR1}\cite{Che_2024.1} A motion $\pi(t,x)$ of dynamical
system $(X,\mathbb S,\pi)$ is two-sided remotely almost periodic
if and only if its dynamically limit set
$\Delta_{x}:=\alpha_{x}\bigcup \omega_{x}$ is equi-almost
periodic.
\end{theorem}

\begin{lemma}\label{lEAP2}\cite{Che_2024.1} Assume that $K_{1},K_{2},\dots,K_{m}$
are the equi-almost periodic compact invariant subsets of
$(X,\mathbb T,\pi)$. Then the following statements hold:
\begin{enumerate}
\item the compact invariant subset $K:=K_1\times K_2\times \ldots
K_m$ of the product dynamical system $(X^{m},\mathbb T,[\pi])$
($X^{m}:=X\times X\times \ldots \times X$ and
$[\pi](t,x):=(\pi(t,x_1),\pi(t,x_2),\ldots,\pi(t,x_m))$ for all
$t\in \mathbb T$ and $x:=(x_1,x_2,\ldots,x_m)\in X^{m}$) is
equi-almost periodic; \item the compact invariant subset
$K:=\bigcup_{i=1}^{m} K_{i}$ of $(X,\mathbb T,\pi)$ is equi-almost
periodic; \item the compact invariant subset $K :=\bigcap
_{i=1}^{m}K_{i}$ of $(X,\mathbb T,\pi)$ is equi-almost periodic.
\end{enumerate}
\end{lemma}

Let $(X,\mathbb T,\pi)$ (respectively, $(Y,\mathbb T,\sigma)$) be
a dynamical system on the space $X$ (respectively, on the space
$Y$) and $x\in X$ ($y\in Y$). Denote by $\mathfrak N_{x}$
(respectively, $\mathfrak M_{x}$ or $\mathfrak L^{+\infty}_{x}$)
the family of all sequence $\{t_n\}\subset \mathfrak T$ such that
$\pi(t_n,x)\to x$ (respectively, the sequence $\{\pi(t_n,x)\}$
converges or $\{\pi(t_n,x)\}$ converges and $t_n\to +\infty$) as
$n\to \infty$.

Let $(Y,\mathbb T,\sigma)$ be a dynamical system on the space $Y$
and $y\in Y$.

\begin{definition}\label{defB1} \cite[Ch.I]{Che_2024} A point $x\in X$
is said to be:
\begin{enumerate}
\item comparable by the character of recurrence with the point $y$
if $\mathfrak N_{y}\subseteq \mathfrak N_{x}$; \item strongly
comparable by character of recurrence with the point $y$ if
$\mathfrak M_{y}\subseteq \mathfrak M_{x}$; \item remotely
comparable by character of recurrence with the point $y$ if
$\mathfrak L_{y}^{+\infty}\subseteq \mathfrak L_{x}^{+\infty}$.
\end{enumerate}
\end{definition}

\begin{theorem}\label{thRAP4.01}\cite[Ch.I]{Che_2024} The following statements are
equivalents:
\begin{enumerate}
\item a point $x\in X$ is strongly comparable by the character of
recurrence with the point $y\in Y$; \item there exists a
continuous mapping $h:H(y)\to H(x)$ satisfying the conditions
\begin{equation}\label{eqH1}
h(y)=x\ \ \mbox{and}\ \ \ h(\sigma(t,q))=\pi(t,h(q)) \nonumber
\end{equation}
for all $(t,q)\in \mathbb T\times H(y)$.
\end{enumerate}
\end{theorem}

\begin{theorem}\label{thRAP4.0} \cite[Ch.I]{Che_2024} Let $y\in Y$ be asymptotically stationary
(respectively, asymptotically $\tau$-periodic or asymptotically
almost periodic) point. If the point $x\in X$ is remotely
comparable by the character of recurrence with the point $y$, then
the point $x$ is also asymptotically stationary (respectively,
asymptotically $\tau$-periodic or asymptotically almost periodic).
\end{theorem}

\begin{theorem}\label{thRAP4.1} \cite[Ch.I]{Che_2024} The
following statements hold:
\begin{enumerate}
\item Let $y\in Y$ be stationary (respectively, $\tau$-periodic or
almost recurrent) point. If the point $x\in X$ is comparable by
the character of recurrence with the point $y$, then the point $x$
is also stationary (respectively, $\tau$-periodic or almost
recurrent); \item Let $y\in Y$ be almost periodic (respectively,
recurrent) point. If the point $x\in X$ is strongly comparable by
the character of recurrence with the point $y$, then the point $x$
is also almost periodic (respectively, recurrent).
\end{enumerate}
\end{theorem}

\begin{theorem}\label{thRAP4} \cite{Che_2024.1} Let $y\in Y$ be Lagrange stable and remotely stationary
(respectively, remotely $\tau$-periodic or remotely almost
periodic) point. If the point $x\in X$ is remotely comparable by
the character of recurrence with the point $y$, then the point $x$
is also remotely stationary (respectively, remotely
$\tau$-periodic or remotely almost periodic).
\end{theorem}

\subsection{Remotely Almost Periodic Functions}\label{Sec2.2}

Let $X$ be a complete metric space. Denote by $C(\mathbb T,X)$ the
space of all continuous functions $\varphi :\mathbb T\to X$
equipped with the distance
\begin{equation*}\label{eqD1}
d(\varphi,\psi):=\sup\limits_{L>0}\min\{\max\limits_{|t|\le L,\
t\in \mathbb T}\rho(\varphi(t),\psi(t)),L^{-1}\}.
\end{equation*}
The space $(C(\mathbb T,X),d)$ is a complete metric space (see,
for example, \cite[ChI]{Che_2020}).

\begin{remark}\label{remD1} \rm
1. The distance $d$ generates on $C(\mathbb R,X)$ the compact-open
topology.

2. The following statements are equivalent:
\begin{enumerate}
\item $d(\varphi_n,\varphi)\to 0$ as $ n\to \infty$; \item
$\lim\limits_{n\to \infty}\max\limits_{|t|\le L,\ t\in \mathbb
T}\rho(\varphi_n(t),\varphi(t))=0$ for each $L>0$; \item there
exists a sequence $l_n\to +\infty$ such that
$$
\lim\limits_{n\to \infty}\max\limits_{|t|\le l_n,\ t\in \mathbb
T}\rho(\varphi_n(t),\varphi(t))=0.
$$
\end{enumerate}
\end{remark}

Let $h\in \mathbb T$, $\varphi \in \mathbb C(\mathbb T,X)$ and
$\varphi^{h}$ be the $h$-translation, i.e.,
$\varphi^{h}(t):=\varphi(t+h)$ for all $t\in \mathbb T$. Denote by
$\sigma_{h}$ the mapping from $C(\mathbb T,X)$ into itself defined
by equality $\sigma_{h}\varphi :=\varphi^{h}$ for every $\varphi \in
C(\mathbb T,X)$. Note that $\sigma_{0}=Id_{C(\mathbb T,X)}$ and
$\sigma_{h_1}\sigma_{h_2}=\sigma_{h_1+h_2}$ for all $h_1,h_2\in
\mathbb T$.

\begin{lemma}\label{lAPF1}\cite[Ch.I]{Che_2020} The mapping $\sigma :\mathbb T\times C(\mathbb T,X)\to C(\mathbb
T,X)$ defined by $\sigma(h,\varphi)=\sigma_{h}\varphi$ for all
$(h,\varphi)\in \mathbb T\times C(\mathbb T,\mathfrak B)$ is
continuous.
\end{lemma}

\begin{coro}\label{corAPF1} The triplet $(C(\mathbb T,X),\mathbb
T,\sigma)$ is a dynamical system (shift dynamical system or
Bebutov's dynamical system).
\end{coro}

Let $\mathfrak B$ be a  Banach space over the field $P$
($P=\mathbb R$ or $\mathbb C$) with the norm $|\cdot|$,
$\rho(u,v):=|u-v|$ ($u,v\in \mathfrak B$) and $\tau\in \mathbb T$
be a positive number. Denote by $C_{0}(\mathbb T,\mathfrak
B)):=\{\varphi \in C(\mathbb T,\mathfrak B)$ such that
$\lim\limits_{t\to +\infty}|\varphi(t)|=0\}$ and $C_{\tau}(\mathbb
T,\mathfrak B):=\{\varphi \in C(\mathbb T,\mathfrak B)|\
\varphi(t+\tau)=\varphi(t)$ for all $t\in \mathbb T\}$.

\begin{definition}\label{defAPF1} Let $\tau \in \mathbb T$ and $\tau >0$. A function $\varphi \in C(\mathbb
T,X)$is said to be:
\begin{enumerate}
\item asymptotically $\tau$-periodic (respectively, asymptotically
stationary) if there exist $p\in C_{\tau}(\mathbb T,\mathfrak B)$
and $r\in C_{0}(\mathbb T,\mathfrak B)$ such that
$\varphi(t)=p(t)+r(t)$ for all $t\in \mathbb T$; \item remotely
$\tau$-periodic \cite{HPT_2008,Kal_2010} (respectively, remotely
stationary) if
\begin{equation}\label{eqAPF2}
\lim\limits_{t\to +\infty}\rho(\varphi(t+\tau),\varphi(t))=0
\end{equation}
(respectively, remotely $\tau$-periodic for every $\tau
>0$);
\item remotely almost periodic
\cite{Bas_2013,Bas_2015,BSS_2019,RS_1986,Sar_1984} if for every
$\varepsilon
>0$ there exists a relatively dense subset $\mathcal
P(\varepsilon,\varphi)$ such that for every $\tau \in \mathcal
P(\varepsilon,\varphi)$ we have a positive number
$L(\varepsilon,\varphi,\tau)$ so that
\begin{equation}\label{eqAPF2.1}
\rho(\varphi(t+\tau),\varphi(t))<\varepsilon \nonumber
\end{equation}
for all $t\ge L(\varepsilon,\varphi,\tau)$.
\end{enumerate}
\end{definition}

\begin{remark}\label{remAPF_02} The functions with the property
(\ref{eqAPF2}) in the work \cite{HPT_2008} (respectively, in the
work \cite{Kal_2010}) is called $S$-asymptotically $\tau$-periodic
(respectively, $\tau$-periodic at the infinity).
\end{remark}

\begin{remark}\label{remAPF2} Every remotely
$\tau$-periodic function is remotely almost periodic.
\end{remark}

\begin{remark}\label{remAPF1} Every asymptotically
$\tau$-periodic (respectively, asymptotically stationary) function
$\varphi \in C(\mathbb T,\mathfrak B)$ is remotely $\tau$-periodic
\cite{HPT_2008,Kal_2010} (respectively, remotely stationary).
\end{remark}

\begin{definition}\label{defAPF02} A function $\varphi \in C(\mathbb
T,X)$ is said to be Lagrange stable if the motion
$\sigma(t,\varphi)$ is so in the shift dynamical system
$(C(\mathbb T,X),\mathbb T,\sigma)$.
\end{definition}

\begin{lemma}\label{lAPF3.1} \cite{Che_2024.2} Let $\varphi \in C(\mathbb T,X)$. The
following statement are equivalent:
\begin{enumerate}
\item[(a)] the motion $\sigma(t,\varphi)$ generated by the
function $\varphi$ in the shift dynamical system $(C(\mathbb
T,X),\mathbb T,\sigma)$ is remotely almost periodic (respectively,
remotely $\tau$-periodic or remotely stationary); \item[(b)] the
function $\varphi$ is remotely almost periodic (respectively,
remotely $\tau$-periodic or remotely stationary).
\end{enumerate}
\end{lemma}

\begin{definition}\label{defAAP1} A function $\varphi \in C(\mathbb T,\mathfrak
B)$ is said to be asymptotically stationary (respectively,
asymptotically $\tau$-periodic or asymptotically almost periodic)
if there exist functions $p,r\in C(\mathbb T,\mathfrak B)$ such
that
\begin{enumerate}
\item $\varphi(t)=p(t)+r(t)$ for all $t\in \mathbb T$; \item $r\in
C_{0}(\mathbb T,\mathfrak B)$ and $p$ is stationary (respectively,
$\tau$-periodic or almost periodic).
\end{enumerate}
\end{definition}

\begin{lemma}\label{lAPP1} \cite[Ch.I]{Che_2009} The following statements are
equivalent:
\begin{enumerate}
\item the function $\varphi \in C(\mathbb T,\mathfrak B)$ is
asymptotically stationary (respectively, asymptotically
$\tau$-periodic or asymptotically almost periodic); \item the
motion $\sigma(t,\varphi)$ of shift dynamical system $C(\mathbb
T,\mathfrak B),\mathbb T,\sigma)$ is asymptotically stationary
(respectively, asymptotically $\tau$-periodic or asymptotically
almost periodic).
\end{enumerate}
\end{lemma}

\begin{theorem}\label{th1RAP} Let $\varphi \in C(\mathbb T,\mathfrak
B)$ be a Lagrange stable function. The function $\varphi$ is
positively remotely almost periodic (respectively, both positively
and negatively remotely almost periodic if $\mathbb T=\mathbb S$)
if and only if its $\omega$-limit set $\omega_{\varphi}$
(respectively, its dynamically limit set
$\Delta_{\varphi}:=\alpha_{\varphi}\bigcup \omega_{\varphi}$) is
equi-almost periodic.
\end{theorem}
\begin{proof} This statement follows directly from Lemmas
\ref{lRAP1}, \ref{lAPF3.1} and Theorem \ref{thTSR1}.
\end{proof}

\subsection{$S^{p}$ remotely almost periodic functions}\label{Sec2.3}

\subsubsection{Shift dynamical systems on the space
$L_{loc}^p(\mathbb{R};\mathfrak{B};\mu).$}\label{Sec2.3.1}

Let $S\subseteq \mathbb R,\ \mathbb T\in \{\mathbb
R_{+},\mathbb{R}\},\;$ $(\mathbb T,\mathfrak{B};\mu)$ be a space
with measure and $\mu$ is the Radon measure, $\mathfrak{B}$ - is a
Banach space with the norm $|\cdot|$.

\begin{definition}
A function $f:S\to\mathfrak{B}$ is called \cite{shv1} a
step-function\index{step-function} if it takes no more than a
finite number of values. In this case it is called measurable, if
$f^{-1}(\{x\})\in \mathfrak{B}$ for every $x\in\mathfrak{B}$, and
integrable if in addition $\mu(f^{-1}(\{x\})<+\infty$. Then there
is defined
\begin{equation}\label{eq1.5.1}
\int fd\mu=\sum_{x\in\mathfrak{B}}\mu(f^{-1}(\{x\})x.
\end{equation}
The sum in the right hand side of the equality (\ref{eq1.5.1}) is
finite by assumption.
\end{definition}

\begin{definition}
A function $f:S\to\mathfrak{B}$ is said to be measurable
\index{measurable function} if there exists a sequence $\{f_n\}$
of step-functions measurable and such that $f_n(s)\to f(s)$ with
respect to the measure $\mu$ almost everywhere.
\end{definition}

\begin{definition}
A function $f:S\to\mathfrak{B}$ is called integrable, if there
exists a sequence $\{f_n\}$ of step-functions, integrable and such
that for every $n$ the function $\varphi_n(s)=|f_n(s)-f(s)|$ is
integrable and
$$
\lim\limits_{n\to+\infty}\int|f_n(s)-f(s)|d\mu(s)=0.
$$
Then $\int f_nd\mu$ converges in the space $\mathfrak{B}$ and its
limit does not depend on the choice of the approximating sequence
$\{f_n\}$ with the above mentioned properties. This limit is
denoted by $\int fd\mu$ or $\int f(s)d\mu(s)$.
\end{definition}

Let $1\le p\le+\infty$. By $L^p(\mathbb T;\mathfrak{B};\mu)$
\index{$L^p(S;\mathfrak{B};\mu)$} there is denoted the space of
all measurable functions (classes of functions)
$f:S\to\mathfrak{B}$ such that $|f|\in L^p(S;\mathbb{R};\mu)$,
where $|f|(s)=|f(s)|$. The space $L^p(S;\mathfrak{B};\mu)$ is
endowed with the norm
\begin{equation}\label{eq1.5.2}
||f||_{L^p}=\big{(}\int|f(s)|^p
d\mu(s)\big{)}^{1/p}\quad\mbox{and}\quad
||f||_{\infty}=supess|f(s)|.
\end{equation}
$L^p(S;\mathfrak{B};\mu)$ with the norm (\ref{eq1.5.2}) is a
Banach space.

Denote by $L_{loc}^p(\mathbb T;\mathfrak{B};\mu)$
\index{$L_{loc}^p(S;\mathfrak{B};\mu)$} the set of all function
$f:\mathbb{T}\to\mathfrak{B}$ such that $f_l\in L^p([-l,l];$
$\mathfrak{B};$ $\mu)$ for every $l>0$, where $f_l$ is the
restriction of the function $f$ onto $\mathbb T_{l}$, where
$\mathbb T_{l}:=[-l,l]$ (if $\mathbb T =\mathbb R$) and $\mathbb
T_{l}:=[0,l]$ (if $\mathbb T=\mathbb R_{+}$).

In the space $L_{loc}^p(\mathbb{T};\mathfrak{B};\mu)$
\index{$L_{loc}^p(\mathbb{T};\mathfrak{B};\mu)$} we define a
family of semi-norms $||\cdot||_{\ell,p}$ by the following rule:
\begin{equation}\label{eq1.5.3}
||f||_{\ell,p}:=||f_{\ell}||_{L^p(\mathbb
T_{\ell};\mathfrak{B};\mu)} \quad(\ell >0).
\end{equation}
Family of the semi-norms (\ref{eq1.5.3}) defines a metrizable
topology on $L_{loc}^p(\mathbb{T};\mathfrak{B};\mu)$. The metric
that gives this topology can be defined, for instance, by the next
equality
\begin{equation}\label{eq1.5.4}
d_p(\varphi,\psi)=\sup\limits_{\ell
>0}\min\{\big{(}\int\limits_{\mathbb T_{\ell}}|\varphi(s)-\psi(s)|^{p}d\mu(s)\big{)}^{1/p};\frac{1}{\ell}\}.\nonumber
\end{equation}

\begin{remark}\label{remS_01}\cite{Che_2024.3} The following statements are
equivalent:
\begin{enumerate}
\item
\begin{equation}\label{eqS_1}
d_{p}(\varphi_n,\varphi)\to 0 \ \ \mbox{as}\ \ n\to \infty
;\nonumber
\end{equation}
\item
\begin{equation}\label{eqS_02}
\lim\limits_{n\to \infty}d_{\ell,p}(\varphi_n,\varphi)=0 \nonumber
\end{equation}
for every $\ell>0$, where $d_{\ell,p}(\varphi,\psi):=\|\varphi -\psi
\|_{\ell,p}$.
\end{enumerate}
\end{remark}

Let us define a mapping
$\sigma:L_{loc}^p(\mathbb{T};\mathfrak{B};\mu)\times\mathbb{T}\to
L_{loc}^p(\mathbb{T};\mathfrak{B};\mu)$ as follows:
$\sigma(\tau,f)=f^{\tau}$ for all $f\in
L_{loc}^p(\mathbb{T};\mathfrak{B};\mu)$ and $\tau\in\mathbb{T}$,
where $f^{\tau}(s):=f(s+\tau)$ $(s\in\mathbb{T})$. It easy to
check that
\begin{enumerate}
\item $\sigma(0,f)=f$ for every $f\in
L_{loc}^p(\mathbb{T};\mathfrak{B};\mu)$; \item
$\sigma(t+\tau,f)=\sigma(t,\sigma(\tau,f))$ for all $t,\tau\in
\mathbb T$ and $f\in L_{loc}^p(\mathbb{T};\mathfrak{B};\mu)$.
\end{enumerate}

\begin{lemma}\label{l1.5.2}\cite[Ch.I]{Che_2009} The mapping
$\sigma :\mathbb T\times L_{loc}^p(\mathbb{T};\mathfrak{B};\mu)
\to L_{loc}^p(\mathbb{T};\mathfrak{B};\mu)$ is continuous and,
consequently, the triplet
$(L_{loc}^p(\mathbb{T};\mathfrak{B};\mu),\mathbb{T},\sigma)$
\index{$(L_{loc}^p(\mathbb{T};\mathfrak{B};\mu),\mathbb{T},\sigma)$}
is a dynamical system on the space
$L_{loc}^p(\mathbb{T};\mathfrak{B};\mu)$.
\end{lemma}

\begin{lemma}\label{lS1}\cite{Che_2024.3} Let $\varphi \in
L_{loc}^p(\mathbb{T};\mathfrak{B};\mu)$ and $\mathbb
T_{\ell}:=[-\ell,\ell]$ ($\ell>0$) if $\mathbb T=\mathbb R$ and
$\mathbb T_{\ell}=[0,\ell]$ if $\mathbb T=\mathbb R_{+}$. The
following conditions are equivalent:
\begin{enumerate}
\item
\begin{equation}\label{eqS_01}
\lim\limits_{t\to +\infty}\int_{0}^{1}|\varphi(t+s)|^{p}d\mu(s)=0;
\end{equation}
\item
\begin{equation}\label{eqS_2}
\lim\limits_{t\to +\infty}\int\limits_{\mathbb
T_{\ell}}|\varphi(t+s)|^{p}d\mu(s)=0 \nonumber
\end{equation}
for every $\ell >0$.
\end{enumerate}
\end{lemma}

\begin{definition}\label{defV1} We will say that a function $\varphi \in L^{p}_{loc}(\mathbb T,\mathfrak
B)$ is $S^{p}$ vanishing at the $+\infty$ if $\lim\limits_{h\to
+\infty}\varphi^{h}=\theta$, where $\theta$ is the null element of
$L^{p}_{loc}(\mathbb T,\mathfrak B)$, i.e., if
\begin{equation}\label{eqV_01}
\lim\limits_{t\to +\infty}\int\limits_{\mathbb
T_{\ell}}|\varphi(t+s)|^{p}d\mu(s)=0 \nonumber
\end{equation}
for every $\ell >0$.
\end{definition}

\begin{remark}\label{rem_V1} A function $\varphi \in L^{p}_{loc}(\mathbb T,\mathfrak
B)$is vanishing at the $+\infty$ if and only if the relation
(\ref{eqS_01}) holds.
\end{remark}

This statement follows from the Definition \ref{defV1} and Lemma
\ref{lS1}.

\subsubsection{Stepanov remotely almost periodic
functions}\label{Sec2.3.2}

Recall \cite{Che_2024.3} that a function $\varphi \in
L^{p}_{loc,\tau}(\mathbb T,\mathfrak B)$ is said to be:
\begin{enumerate}
\item $S^{p}$ $\tau$-periodic if $\varphi^{\tau}=\varphi$, i.e.,
$\int_{0}^{1}|\varphi(t+\tau +s)-\varphi(t+s)|^{p}d\mu(s)=0$ for
all $t\in \mathbb R$; \item $S^{p}$ stationary if it is
$\tau$-periodic for every $\tau \in \mathbb T$.
\end{enumerate}

Denote by
\begin{enumerate}
\item $L^{p}_{loc,\tau}(\mathbb T,\mathfrak B)$ the space of all
$\tau$-periodic function $\varphi$ from $L^{p}_{loc}(\mathbb
T,\mathfrak B)$; \item $L^{p}_{loc,st.}(\mathbb T,\mathfrak B)$
the space of all stationary functions from $L^{p}_{loc}(\mathbb
T,\mathfrak B)$; \item $L^{p}_{loc,0}(\mathbb T,\mathfrak B)$ the
space of all functions $\varphi$ from $L^{p}_{loc}(\mathbb
T,\mathfrak B)$ $S^{p}$-vanishing at the infinity, i.e., if
\begin{equation}\label{eqV1}
\lim\limits_{t\to +\infty}\int_{0}^{1}|\varphi(t+s)|^{p}d\mu(s)
=0.\nonumber
\end{equation}
\end{enumerate}

Let $\varphi \in L^{p}_{loc}(\mathbb T;\mathfrak B;\mu)$. Denote
by $\hat{\varphi}$ the mapping from $\mathbb T$ to
$L^{p}([0,1];\mathfrak B;\mu)$ defined by
$\hat{\varphi}(t):=\varphi^{t}{|_{[0,1]}}$, i.e.,
\begin{equation}\label{eqP1}
\big{(}\hat{\varphi}(t)\big{)}(s):=\varphi(t+s) \nonumber
\end{equation}
for any $s\in [0,1]$ and $t\in \mathbb T$. It is clear that
$\hat{\varphi}(t)\in L^{p}([0,1];\mathfrak B;\mu)$ because
$\varphi \in L^{p}_{loc}(\mathbb T;\mathfrak B;\mu)$.

\begin{definition}\label{defSAPF1} A function $\varphi \in L^{p}_{loc}(\mathbb
T,X)$ is said to be:
\begin{enumerate}
\item $S^{p}$ almost periodic if the function $\hat{\varphi}\in
C(\mathbb T,L^{p}([0,1],\mathfrak B))$ is almost periodic, that
is, for every $\varepsilon >0$ there exists a relatively dense in
$\mathbb T$ subset $\mathcal F(\varepsilon,\varphi)\subset \mathbb
T$ such that
\begin{equation}\label{eqSAP1}
\int_{0}^{1}|\varphi(t+\tau
+s)-\varphi(t+s)|^{p}d\mu(s)<\varepsilon^{p}\nonumber
\end{equation}
for all $t\in \mathbb T$ and $\tau \in \mathcal
F(\varepsilon,\varphi)$; \item $S^{p}$ asymptotically $\tau $
($\tau\in \mathbb T$, $\tau >0$)-periodic (respectively,
asymptotically stationary) if there exist $\textbf{p}\in
L^{p}_{loc,\tau}(\mathbb T,\mathfrak B)$ (respectively,
$\textbf{p}\in L^{p}_{loc,st.}(\mathbb T,\mathfrak B)$ ) and $r\in
L^{p}_{loc,0}(\mathbb T,\mathfrak B)$ such that
$\varphi(t)=\textbf{p}(t)+r(t)$ for all $t\in \mathbb T$; \item
$S^{p}$ remotely $\tau$-periodic (respectively, $S^{p}$ remotely
stationary) if
\begin{equation}\label{eqV2}
\lim\limits_{t\to
+\infty}\int_{0}^{1}|\varphi(t+\tau+s)-\varphi(t+s)|^{p}d\mu(s)
=0.\nonumber
\end{equation}
(respectively, $S^{p}$ remotely $\tau$-periodic for every $\tau
>0$);
\item $S^{p}$ remotely almost periodic if for every $\varepsilon
>0$ there exists a relatively dense in $\mathbb T$ subset $\mathcal
P(\varepsilon,\varphi)$ such that for all $\tau \in \mathcal
P(\varepsilon,\varphi)$ we have a positive number
$L(\varepsilon,\varphi,\tau)$ so that
\begin{equation}\label{eqV3}
\int_{0}^{1}|\varphi(t+\tau
+s)-\varphi(t+s)|^{p}d\mu(s)<\varepsilon^{p} \nonumber
\end{equation}
for all $t\ge L(\varepsilon,\varphi,\tau)$.
\end{enumerate}
\end{definition}

\begin{remark}\label{remV1} Every $S^{p}$ remotely
$\tau$-periodic function is $S^{p}$ remotely almost periodic.
\end{remark}

\begin{remark}\label{remV2} Every $S^{p}$ asymptotically
$\tau$-periodic (respectively, $S^{p}$ asymptotically stationary)
function $\varphi \in L^{p}_{loc}(\mathbb T,\mathfrak B)$ is
$S^{p}$ remotely $\tau$-periodic (respectively, $S^{p}$ remotely
stationary).
\end{remark}

\begin{definition}\label{defAPF2} A function $\varphi \in L^{p}_{loc}(\mathbb
T,\mathfrak B)$ is said to be Lagrange stable if the motion
$\sigma(t,\varphi)$ is so in the shift dynamical system
$(L^{p}_{loc}(\mathbb T,\mathfrak B),\mathbb T,\sigma)$.
\end{definition}

\begin{lemma}\label{lAPF2} \cite[Ch.III]{Sel_1971}  Assume that the Banach space $\mathfrak B$ is finite dimensional.
A function $\varphi \in L^{p}_{loc}(\mathbb T,\mathfrak B)$ is
Lagrange stable if and only if the following conditions are
fulfilled:
\begin{enumerate}
\item the set function $\varphi$ is $S^{p}$ bounded, i.e.,
$\sup\{\int_{0}^{1}|\varphi(t+s)|^{p}ds |\ t\in \mathbb
T\}<+\infty$; \item the function $\varphi \in L^{P}_{loc}(\mathbb
T,\mathfrak B)$ is $S^{p}$ uniformly continuous on $\mathbb T$,
i.e., for every $\varepsilon >0$ there exists a positive number
$\delta =\delta(\varepsilon)>0$ such that $|h|<\delta$ implies
\begin{equation}\label{eqV4}
\big{(}\int_{0}^{1}|\varphi(t+h+s)-\varphi(t+s)|^{p}d\mu(s)\big{)}^{1/p}
<\varepsilon \nonumber
\end{equation}
for all $t\in \mathbb T$.
\end{enumerate}
\end{lemma}

\begin{lemma}\label{lAPF3}\cite{Che_2024.3} Let $\varphi \in C(\mathbb T,\mathfrak B)$ and $\tau \in
\mathbb T$ be a positive number. The following statements are
equivalent:
\begin{enumerate}
\item
\begin{equation}\label{eqAPF3}
\lim\limits_{t\to +\infty}\int_{0}^{1}|\varphi(t+\tau
+s)-\varphi(t+s)|^{p}d\mu(s)=0;
\end{equation}
\item
\begin{equation}\label{eqAPF4}
\lim\limits_{t\to +\infty}\int_{0}^{\ell}|\varphi(t+\tau
+s)-\varphi(t+s)|^{p}d\mu(s) =0 \nonumber
\end{equation}
for every $\ell >0$.
\end{enumerate}
\end{lemma}

\begin{coro}\label{corS1} A function $\varphi \in L^{p}(\mathbb T,\mathfrak
B)$ is $S^{p}$ remotely $\tau$-periodic if and only if the
relation (\ref{eqAPF3}) holds.
\end{coro}

\begin{theorem}\label{th1SRAP}\cite{Che_2024.3} Let $\varphi \in L^{p}_{loc}(\mathbb T,\mathfrak
B)$ be a Lagrange stable function. The function $\varphi$ is
positively remotely $S^{p}$ almost periodic (respectively, both
positively and negatively remotely $S^{p}$ almost periodic if
$\mathbb T=\mathbb R$) if and only if its $\omega$-limit set
$\omega_{\varphi}$ (respectively, its dynamically limit set
$\Delta_{\varphi}:=\alpha_{\varphi}\bigcup \omega_{\varphi}$) is
$S^{p}$ equi-almost periodic.
\end{theorem}

\subsubsection{Some properties of $S^{p}$ remotely almost periodic
functions}\label{Sec2.3.3}

Denote by
\begin{enumerate}
\item $S^{p}_{b}(\mathbb T,\mathfrak B):=\{\varphi \in
L^{p}_{loc}(\mathbb T,\mathfrak B)|\ \sup\limits_{t\in \mathbb
T}(\int_{t}^{t+1}|\varphi(s)|^{p}d\mu(s))^{1/p}<+\infty\}$;
 \item $S^{p}RAS(\mathbb T,\mathfrak B)$ the family of all
remotely stationary functions $\varphi \in S^{p}_{b}(\mathbb
T,\mathfrak B)$; \item $S^{p}RAP_{\tau}(\mathbb T,\mathfrak B)$
the family of all remotely $\tau$-periodic functions $\varphi \in
S^{p}_{b}(\mathbb T,\mathfrak B)$; \item $S^{p}RAP(\mathbb
T,\mathfrak B)$ the family of all remotely almost periodic
functions $\varphi \in S^{p}_{b}(\mathbb T,\mathfrak B)$.
\end{enumerate}

\begin{remark}\label{remS1} Note that $S^{p}_{b}(\mathbb T,\mathfrak
B)$ is a Banach space with the norm
\begin{equation}\label{eqS_001}
\|\varphi\|_{S^{p}}:=\sup\limits_{t\in \mathbb
T}\big{(}\int_{t}^{t+1}|\varphi(s)|^{p}d\mu(s)\big{)}^{1/p}.\nonumber
\end{equation}
\end{remark}

\begin{theorem}\label{thRAPS_G} \cite{Che_2024.3} The following statements hold:
\begin{enumerate}
\item[1.] Let $h\in \mathbb T$ and $\varphi \in S^{p}RAP(\mathbb
T,\mathfrak B)$ then $\varphi^{h}\in S^{p}RAP(\mathbb T,\mathfrak
B)$; \item[2.] $S^{p}RAP(\mathbb T,\mathfrak B)$ is a closed
subspace of the space $S^{p}_{b}(\mathbb T,\mathfrak B)$;
\item[3.] For every $f\in S^{p}RAP(\mathbb T,\mathfrak B)$ we have
$H(f)\subset S^{p}RAP(\mathbb T,\mathfrak B)$, where by bar the
closure in the space $S^{p}_{b}(\mathbb T,\mathfrak B)$ is
denoted; \item[4.] Assume that the functions $f_1,\ldots,f_{m}\in
L^{p}_{loc}(\mathbb T,\mathfrak B)$ are Lagrange stable and
$f_{i}\in S^{p}RAP(\mathbb T,\mathfrak B_{i})$ ($i=1,\ldots,m$),
then
\begin{enumerate} \item the
function $F:=(f_1,\ldots,f_{m})\in L^{p}_{loc}(\mathbb T,\mathfrak
B)$ is Lagrange stable, where $\mathfrak B:=\mathfrak B_{1}\times
\ldots \times \mathfrak B_{m}$; \item the function $F\in
S^{p}RAP(\mathbb T,\mathfrak B)$.
\end{enumerate}
\item[5.] Assume that the functions $f_i\in S^{p}RAP(\mathbb
T,\mathfrak B)$ ($i=1,\ldots,m$) and Lagrange stable in the shift
dynamical system $(L^{p}_{loc}(\mathbb T,\mathfrak B),\mathbb
T,\sigma)$ then
\begin{enumerate}
\item the function $F:=f_1+\ldots +f_m$ is also Lagrange stable
and \item $F\in S^{p}RAP(\mathbb T,\mathfrak B)$.
\end{enumerate}
\item[6.] Assume that the following conditions are fulfilled:
\begin{enumerate}
\item the functions $f_i\in S^{p}RAP(\mathbb T,\mathfrak B)$
($i=1,2,\ldots $) and they are Lagrange stable in the shift
dynamical system $(L^{p}_{loc}(\mathbb T,\mathfrak B),\mathbb
T,\sigma)$; \item the functional series
\begin{equation}\label{eq6RAP}
f_1 +f_2+\ldots +f_m +\ldots \nonumber
\end{equation}
converges in the space $S^{p}_{b}(\mathbb T,\mathfrak B)$ and $F$
is its some.
\end{enumerate}

Then the function
$$
F=\sum_{i=1}^{\infty}f_{i}
$$
is Lagrange stable
and $F\in S^{p}RAP(\mathbb T,\mathfrak B)$. \item[7.] Let $\alpha,
\beta \in P$ and $f_1,f_2\in S^{p}RAP(\mathbb T,\mathfrak B)$. If
the functions $f_1$ and $f_2$ are Lagrange stable in the shift
dynamical system $(L^{p}_{loc}(\mathbb T,\mathfrak B),\mathbb
T,\sigma)$, then
\begin{enumerate}
\item the function $F:=\alpha f_1 +\beta f_2$ is Lagrange stable;
\item $F\in S^{p}RAP(\mathbb T,\mathfrak B)$.
\end{enumerate}
\item[8.] Assume that the following conditions hold:
\begin{enumerate}
\item $p,q>0$ and $p^{-1}+q^{-1}=1$; \item the function $\varphi
\in S^{q}_{b}(\mathbb T,P)$ (a scalar function) is Lagrange stable
in the shift dynamical system $(L^{q}_{loc}(\mathbb T,P),\mathbb
T,\sigma)$ (respectively, $f\in S^{p}_{b}(\mathbb T,\mathfrak B)$
is Lagrange stable in the shift dynamical system
$(L^{p}_{loc}(\mathbb T,\mathfrak B),\mathbb T,\sigma)$); \item
$\varphi \in S^{q}RAP(\mathbb T,P)$ and $f\in S^{p}RAP(\mathbb
T,\mathfrak B)$.
\end{enumerate}

Then
\begin{enumerate}
\item the function $F =\varphi f$ is Lagrange stable in the shift
dynamical system $(L^{1}_{loc}(\mathbb T,\mathfrak B),\mathbb
T,\sigma)$ and \item $F\in S^{1}RAP(\mathbb T,\mathfrak B)$.
\end{enumerate}
\end{enumerate}
\end{theorem}

\begin{remark}\label{remS9RAP} Using the arguments above we can
show that Theorem \ref{thRAPS_G} remains true if we replace
everywhere the space $S^{p}RAP(\mathbb T,\mathfrak B)$ by
$S^{p}RAS(\mathbb T,\mathfrak B)$ (respectively,
$S^{p}RAP_{\tau}(\mathbb T,\mathfrak B)$).
\end{remark}

\section{Integration of $S^{p}$ remotely almost periodic
functions}\label{Sec3}

Consider a linear nonhomogeneous differential equation
\begin{equation}\label{eqISP1}
 x'=Ax+f(t)
\end{equation}
on the space $\mathfrak B$, where $f\in L_{loc}^{p}(\mathbb
T,\mathfrak B)$ and $A$ is an infinitesimal generator which
generates a $C_0$-semigroup $\{U(t)\}_{t\ge 0}$ acting on
$\mathfrak B$.

\begin{lemma}\label{lB1}\cite[Ch.VI]{Che_2020} Let $\{U(t)\}_{t\ge 0}$ be a
$C_0$-semigroup acting on the space $\mathfrak B$. Then for every
$L>0$ there exists a  positive constant $C=C(L)$ such that
$\|U(t)\|\le C$ for all $t\in [0,L]$.
\end{lemma}

Recall that a continuous function $u:\mathbb T \to \mathfrak B$ is
a weak (mild) solution of the equation (\ref{eqISP1}) if for every
$t_0\in \mathbb T$ we have
\begin{equation}\label{eqISP2}
u(t)=U(t-t_0)u(t_0)+\int_{t_0}^{t}U(t-s)f(s)ds \nonumber
\end{equation}
for all $t\ge t_0$.

\begin{lemma}\label{lISP1} Let $Q$ be a compact subset from $\mathfrak
B$, $\{\varphi_n\}$ be a sequence from $C(\mathbb T,Q)$ and
$\{f_n\}$ be a sequence from $L^{p}_{loc}(\mathbb T,\mathfrak B)$.
Assume that the function $\varphi_{n}$ is a solution of the equation
\begin{equation}\label{eqISP3}
x'=Ax+f_{n}(t)\nonumber
\end{equation}
for all $n\in \mathbb N$.

Then
\begin{enumerate}
\item if the sequence $\{f_n\}$ is precompact in
$L^{p}_{loc}(\mathbb T,\mathfrak B)$, then the sequence
$\{\varphi_n\}$ is precompact in $C( \mathbb T,Q)$; \item if the
sequence $\{f_n\}$ converges to $f$ as $n\to \infty$ in the space
$L^{p}_{loc}(\mathbb T,\mathfrak B)$, then
\begin{enumerate}
\item the sequence $\{\varphi_n\}$ is precompact in $C(\mathbb
T,Q)$ and \item every limiting function $\varphi$ of the sequence
$\{\varphi_{n}\}$ is a solution of the equation (\ref{eqISP1}).
\end{enumerate}
\item if the sequence $\{f_n\}$ converges $f$ as $n\to \infty$ in
the space $L^{p}_{loc}(\mathfrak B)$, $t_0\in \mathbb T$ such that
$\lim\limits_{n\to \infty}\varphi_{n}(t_0)=x_0$ and $\varphi$ is a
solution of (\ref{eqISP1}) with the initial data
$\varphi(t_0)=x_0$, then the sequence $\{\varphi_n\}$ converges to
$\varphi$ as $n\to \infty$ in the space $C(\mathbb T,Q)$.
\end{enumerate}
\end{lemma}
\begin{proof}
Assume that the sequence $\{f_n\}$ is precompact in the space
$L^{p}_{loc}(\mathbb T,\mathfrak B)$. Without loss of generality
we can suppose that the sequence $\{f_n\}$ converges in
$L^{p}_{loc}(\mathbb T,\mathfrak B)$. Then we have
\begin{equation}\label{eqISP4}
\lim\limits_{k,m\to \infty}\int\limits_{|t|\le
L}|f_k(t)-f_{m}(t)|^{\alpha}d t=0
\end{equation}
for every $L>0$ and $1\le \alpha \le p$. Since $\varphi_n(t_0)\in Q$
(for all $n\in \mathbb N$) and $Q$ is a compact subset of
$\mathfrak B$ then we can assume that the sequence
$\{\varphi_n(t_0)\}$ also converges and, consequently,
\begin{equation}\label{eqISP5}
\lim\limits_{k,m\to \infty}|\varphi_k(t_0)-\varphi_{m}(t_0)|=0.
\end{equation}
Note that
\begin{equation}\label{eqB1}
\varphi_{k}(t)-\varphi_{m}(t)=U(t-t_0)(\varphi_{k}(t_0)-\varphi_{m}(t_0))+\int_{t_0}^{t}U(t-s)(\varphi_{k}(s)-\varphi_{m}(s))ds
\end{equation}
for all $k,m\in \mathbb N$. From (\ref{eqB1}) taking into account
Lemma \ref{lB1} we obtain
$$
\max\limits_{|t|\le L}|\varphi_{k}(t)-\varphi_{m}(t)|\le
$$
\begin{equation}\label{eqISP6}
\big{(}|\varphi_{k}(t_0)-\varphi_{m}(t_0)|+\int\limits_{|t|\le
2L}|f_{k}(s)-f_{m}(s)|ds\big{)}C(2L)
\end{equation}
for every $L>0$. Passing to the limit in (\ref{eqISP6}) as $k,m\to
\infty$ and taking into account (\ref{eqISP4}) and (\ref{eqISP5}) we
obtain
\begin{equation}\label{eqISP7}
\lim\limits_{k,m\to \infty}\max\limits_{|t|\le
L}|\varphi_{k}(t)-\varphi_{m}(t)|=0
\end{equation}
for every $L>0$. Since the space $C(\mathbb T,\mathfrak B)$ is
complete then from the relation (\ref{eqISP7}) follows the
convergence of the sequence $\{\varphi_{n}\}$ in the space
$C(\mathbb T,\mathfrak B)$.

Let now $\lim\limits_{n\to \infty}f_n=f$ in the space
$L^{p}_{loc}(\mathbb T,\mathfrak B)$. According to the first
statement of Lemma the sequence $\{\varphi_{n}\}$ is precompact in
$C(\mathbb T,\mathfrak B)$. Let $\widetilde{\varphi}$ be a limiting
function for the sequence $\{\varphi_{n}\}$ and
$\{\varphi_{n_k}\}$ be a subsequence of $\{\varphi_{n}\}$ such
that $\lim\limits_{k\to \infty}\varphi_{n_k}=\widetilde{\varphi}$. We
will show that $\widetilde{\varphi}$ is a solution of the equation
(\ref{eqISP1}). Assume that $t>t_0$. We will choose $L$
sufficiently large (so that $t,t_0\in [-L,L]$). Then we have
$$
|\widetilde{\varphi}(t)-U(t-t_0)\widetilde{\varphi}(t_0)-\int_{t_0}^{t}U(t-s)\widetilde{f}(s)ds|=
$$
$$
|\widetilde{\varphi}(t)-\varphi_{n_k}(t)+U(t-t_0)(\varphi_{n_k}(t_0)-\widetilde{\varphi}(t_0))+
$$
$$
\int_{t_0}^{t}U(t-s)(f_{n_k}(s)-\widetilde{f}(s))ds|\le
$$
$$
|\varphi_{n_k}(t)-\widetilde{\varphi}(t)|+C(2L)\big{(}|\varphi_{n_k}(t_0)-\widetilde{\varphi}(t_0)|
+ \int_{t_0}^{t}|f_{n_k}(s)-f(s)|ds\big{)}.
$$
Thus we have
\begin{equation}\label{eqISP10}
|\widetilde{\varphi}(t)-U(t-t_0)\widetilde{\varphi}(t_0)-\int_{t_0}^{t}U(t-s)\widetilde{f}(s)ds|
\le
\end{equation}
$$
\max\limits_{|t|\le
L}|\varphi_{n_k}(t)-\widetilde{\varphi}(t)|+C(2L)\big{(}|\varphi_{n_k}(t_0)-\widetilde{\varphi}(t_0)|
+ \int\limits_{|t|\le 2L}|f_{n_k}(s)-f(s)|ds\big{)}.
$$
Since $f_n\to f$ (in $L^{p}_{loc}(\mathbb T,\mathfrak B)$) and
$\varphi_{n_k}\to \widetilde{\varphi}$ (in $C(\mathbb T,Q)$) then
we obtain
\begin{equation}\label{eqISP9}
\lim\limits_{k\to \infty}\int\limits_{|t|\le
L}|f_{n_k}(s)-f(s)|ds=0 \ \mbox{and}\ \lim\limits_{k\to
\infty}\max\limits_{|t|\le
L}|\varphi_{n_k}(t)-\widetilde{\varphi}(t)|=0
\end{equation}
for every $L>0$.

Passing to the limit in (\ref{eqISP10}) and taking into account
(\ref{eqISP9}) we have
\begin{equation}\label{eqISP11}
\widetilde{\varphi}(t)=U(t-t_0)\widetilde{\varphi}(t_0)+
\int_{t_0}^{t}U(t-s)\widetilde{f}(s)ds,\nonumber
\end{equation}
i.e., the function $\widetilde{\varphi}$ is a solution of the
equation (\ref{eqISP1}).

Let $\lim\limits_{n\to \infty}\varphi_{n}(t_0)=x_0$ and $\varphi$
be a solution of the equation (\ref{eqISP1}) with the initial data
$\varphi(t_0)=x_0$. Since the sequence $\{\varphi_n\}$ is
precompact then it contains a convergent subsequence
$\{\varphi_{n_k}\}$. Let $\widetilde{\varphi}:=\lim\limits_{k\to
\infty}\varphi_{n_k}$ then
$\widetilde{\varphi}(t_0)=\lim\limits_{k\to
\infty}\varphi_{n_k}(t_0)=x_0$. From the last relation follows
that the solutions $\varphi$ and $\widetilde{\varphi}$ coincide.
Additionally, from the above follows that the limit of all
subsequence of the sequence $\{\varphi_n\}$ coincides with
$\varphi$. This means that the sequence $\{\varphi_n\}$ converges
to $\varphi$ (in the space $C(\mathbb T,Q)$) as $n\to \infty$.
Lemma is completely proved.
\end{proof}

\begin{coro}\label{corISP1} Assume that the function $f\in
L^{p}_{loc}(\mathbb T, \mathfrak B)$ is Lagrange stable
(respectively, positively Lagrange stable) in $(L^{p}_{loc}(\mathbb
T,\mathfrak B),\mathbb T,\sigma)$. Then every compact solution
$\varphi$ of (\ref{eqISP1}) is Lagrange stable (respectively,
positively Lagrange stable) in $(C(\mathbb T,\mathfrak B),\mathbb
T,\sigma)$.
\end{coro}

\begin{coro}\label{corISP2} If the function $f\in
L^{p}_{loc}(\mathbb T, \mathfrak B)$ is Lagrange stable
(respectively, positively Lagrange stable) in $(L^{p}_{loc}(\mathbb
T,\mathfrak B),\mathbb T,\sigma)$, then every compact primitive
$\Phi$ of $\varphi$ is Lagrange stable (respectively, positively
Lagrange stable) in $(C(\mathbb T,\mathfrak B),\mathbb T,\sigma)$.
\end{coro}

\begin{definition}\label{defPS1} A function $\varphi \in L^{p}_{loc}(\mathbb T,\mathfrak
B)$ is said to be $S^{p}$ Poisson stable if the motion
$\sigma(t,\varphi)$ is Poisson stable in the shift dynamical
system $(L^{p}_{loc}(\mathbb T,\mathfrak B),\mathbb T,\sigma)$.
\end{definition}

\begin{theorem}\label{thP0} Let a function $\varphi \in L^{p}_{loc}(\mathbb T,\mathfrak
B)$ be $S^{p}$ Poisson stable and $F$ be a bounded primitive of
$\varphi$. Assume that $F$ is weakly precompact, i.e., the set
$\{F(t_n)\}$ is precompact for every $\{t_n\}\in \mathfrak
N_{\psi}$. Then the primitive $F$ of the function $\varphi$ is
comparable by the character of recurrence with $\varphi$, that is,
$\mathfrak N_{\varphi}\subseteq \mathfrak N_{F}$.
\end{theorem}
\begin{proof}
This statement was proved by Shcherbakov
\cite[Ch.II]{Sch72},\cite{Sch73} if the function $\varphi$ is
continuous and Poisson stable (in the shift dynamical system
($C(\mathbb T,\mathfrak B),\mathbb T,\sigma)$). By some
modification of the ideas of Shcherbakov we will present the proof
for $S^{p}$ (Stepanov) Poisson stable functions $\varphi \in
L^{p}_{loc}(\mathbb T,\mathfrak B)$.

Consider the primitive $F_0$ of $\varphi$ defined by the equality
\begin{equation}\label{eqPF1}
F_0(t):=\int_{0}^{t}\varphi(s)ds \nonumber
\end{equation}
for all $t\in \mathbb T$.

We will show that under the conditions of Theorem if $\{t_n\}\in
\mathfrak N_{\psi}$ and the sequence $\{F(t_n)\}$ converges then
$\lim\limits_{n\to \infty}F(t_n)=0$, where $0$ is the null element
of $\mathfrak B$.

Assume that
\begin{equation}\label{eqPF2}
\lim\limits_{n\to \infty}F(t_n)=x_0\not= 0
\end{equation}
and consider the sequence of functions $\{F_n\}$ from $C(\mathbb
T,\mathfrak B)$ defined by the equality
\begin{equation}\label{eqPF3}
F_n(t):=nx_0+\int_{0}^{t}\varphi(s)ds \nonumber
\end{equation}
for all $n\in \mathbb N$ and $t\in \mathbb T$. We will show that
for an $k\in \mathbb N$ we have
\begin{equation}\label{eqPF4}
\overline{F_{k}(\mathbb T)} \subset \overline{F_{k-1}(\mathbb
T)}.\nonumber
\end{equation}
Indeed, we note that
\begin{equation}\label{eqPF5}
F_{k}(t_0)-F_{k-1}(t_0+t_n)=x_0-F_0(t_n)+\int_{0}^{t_0}[\varphi(s)-\varphi(s+t_n)]ds
.
\end{equation}
Passing to the limit in (\ref{eqPF5}) as $n\to \infty$ and taking
into account (\ref{eqPF2}) and $\{t_n\}\in \mathfrak N_{\varphi}$
(i.e., $\varphi^{t_n}\to \varphi$ as $n\to \infty$ in the space
$L^{p}_{loc}(\mathbb T,\mathfrak B)$) by Lemma \ref{lISP1} (item
(iii)) we obtain
\begin{equation}\label{eqPF6}
F_{k}(t_0)=\lim\limits_{n\to \infty}F_{k-1}(t_0+t_n) .\nonumber
\end{equation}
This means that $F_{k}(t_0)\in \overline{F_{k-1}(\mathbb T)}$ for
every $t_0\in \mathbb T$. Thus we have $F_{k}(\mathbb T)\subset
\overline{F_{k-1}(\mathbb T)}$ and, consequently,
$\overline{F_{k}(\mathbb T)}\subset \overline{F_{k-1}(\mathbb T)}$
and
\begin{equation}\label{eqPF7}
\overline{F_{0}(\mathbb T)}\supset \overline{F_{1}(\mathbb
T)}\supset \ldots \supset \overline{F_{n}(\mathbb T)}\supset
\ldots .\nonumber
\end{equation}
In particular, we have $\overline{F_{n}(\mathbb T)}\subset
\overline{F_{0}(\mathbb T)}$ for all $n\in \mathbb N$. But
\begin{equation}\label{eqPF8}
nx_0\in \overline{F_{n}(\mathbb T)}\subset \overline{F_{0}(\mathbb
T)}
\end{equation}
for every $n\in \mathbb N$. The relation (\ref{eqPF8}) contradicts
to the boundedness of the primitive $F_0$.

Now we will show that the bounded primitive $F$ of the function
$\varphi$ is comparable with the function $\varphi$. Indeed, let
$\{t_n\}\in \mathfrak N_{\varphi}$ then we can extract a
subsequence $\{t_{k_{n}}\}$ such that
\begin{equation}\label{eqPF9}
\lim\limits_{n\to \infty}F(t_{k_n})=0 .
\end{equation}
Note that
\begin{equation}\label{eqPF10}
F(t+t_{k_n})-F(t)=F(t_n)+\int_{0}^{t}[\varphi(s+t_{k_n})-\varphi(s)]ds
.
\end{equation}
Passing to the limit in (\ref{eqPF10}) as $n\to \infty$ and taking
into account (\ref{eqPF9}) and the fact that the sequence
$\{\varphi^{t_n}\}$ converges to $\varphi$ in the space
$L^{p}_{loc}(\mathbb T,\mathfrak B)$ by Lemma \ref{lISP1} (item
(iii)) we obtain $F^{t_{k_{n}}}\to F$ as $n\to \infty$ in the
space $C(\mathbb T,\mathfrak B)$. This means that $\{t_{k_n}\}\in
\mathfrak N_{F}$. Thus the sequence $\{F^{t_n}\}$ is precompact
and it has a unique limiting point $F$ in the space $C(\mathbb
T,\mathfrak B)$ and, consequently, $\{t_n\}\in \mathfrak N_{F}$.
Theorem is proved.
\end{proof}

\begin{coro}\label{corP1} Let $\varphi \in L^{p}_{loc}(\mathbb T,\mathfrak
B)$be a $S^{p}$ Poisson stable function. Every precompact
primitive of the function $\varphi$ is comparable by the character
of recurrence with the function $\varphi$.
\end{coro}
\begin{proof} This statement directly follows from Theorem
\ref{thP0} because if the primite $F$ of the function $\varphi$ is
precompact, then it is evidently weakly precompact.
\end{proof}

\begin{coro}\label{corP2} Let $\varphi \in L^{p}_{loc}(\mathbb T,\mathfrak
B)$ be a $S^{p}$ stationary (respectively, $S^{p}\ \tau$-periodic,
$S^{p}$ almost recurrent). Every precompact primitive of the
function $\varphi$ is stationary (respectively, $S^{p}$
$\tau$-periodic, $S^{p}$ almost recurrent).
\end{coro}
\begin{proof} This statement follows from Corollary \ref{corP1}
and Theorem \ref{thRAP4.1}.
\end{proof}

\begin{theorem}\label{thP1} Let function $\varphi \in L^{p}_{loc}(\mathbb T,\mathfrak
B)$ be recurrent. Every precompact primitive of the function
$\varphi$ is strongly comparable by the character of recurrence
with the function $\varphi$.
\end{theorem}
\begin{proof} This statement was proved by Shcherbakov \cite[Ch.II]{Sch72},\cite{Sch73}
if the function $\varphi$ is continuous and recurrent (in the
shift dynamical system ($C(\mathbb T,\mathfrak B),\mathbb
T,\sigma)$).

Denote by $Y:=H(\varphi)$ the closure in the space
$L^{p}_{loc}(\mathbb T,\mathfrak B)$ the family of all shifts
$\{\varphi^{h}|\ h\in \mathbb T\}$. Under the condition of Theorem
$Y$ is a compact and minimal set of the shift dynamical system
$(L^{p}_{loc}(\mathbb T,\mathfrak B),\mathbb T,\sigma)$. Let
$X:=\mathfrak B \times Y$ and $(X,\mathbb T,\sigma)$ be the
skew-product dynamical system $(X,\mathbb T,\pi)$, where
$\pi(h,(v,g))):=(\phi(h,v,g),g^{h})$ for all $(t,v,g)\in \mathbb
T\times \mathfrak B\times H(\varphi)$,
$\phi(t,v,g):=v+\int_{0}^{t}g(s)ds$ and $g\in H(\varphi$. Denote
by $\langle (X,\mathbb T,\pi),(Y,\mathbb T,\sigma),h\rangle$
($h:=pr_2:X\to Y$) the non-autonomous dynamical system associated
by skew-product dynamical system $(X,\mathbb T,\pi)$.

Consider an arbitrary compact primitive $F$ of the function
$\varphi$ then there exists a point $u\in \mathfrak B$ such that
\begin{equation}\label{eqSP1}
F(t)=\phi(t,u,\varphi)=u+\int_{0}^{t}\varphi(s)ds
\end{equation}
for all $t\in \mathbb T$. Denote by
$X_{0}:=H(u,\varphi)=\overline{\{\pi(t,(u,\varphi))|\ t\in \mathbb
T\}}$, where $\pi(h,(u,\varphi))=(\phi(h,u,\varphi),\varphi^{h})$
for all $h\in \mathbb T$. It is clear that $X_0$ is a compact and
invariant subset of skew-product dynamical system $(X,\mathbb
T,\pi)$. By Birkhoff theorem there exists a compact minimal subset
$M\subset X_0$. Since $Y$ is a compact and minimal set and $h:X\to
Y$ is a homomorphism of $(X,\mathbb T,\pi)$ onto $(Y,\mathbb
T,\sigma)$ then $h(M)=Y$ and, consequently,
$M_{y}:=h^{-1}(y):=\{(v,y)|\ (v,y)\in M\}\not= \emptyset$ for all
$y\in Y$. We will show that the set $M_{y}$ consists of a single
point $\{(u_y,y)\}$ for all $y\in Y$. If we assume that it is not
true, then there exists a point $\varphi_0\in Y=H(\varphi)$ such
that the set $M_{\varphi_0}=\{(u,\varphi_0)|\ (u,\varphi_0)\in
M\}$ contains at least two different points $(u_i,\varphi_0)$
($i=1,2$), i.e., $u_1\not= u_2$ ($u_1,u_2\in \mathfrak B)$. It is
clear (see formula (\ref{eqSP1})) that
\begin{equation}\label{eqSP2}
\phi(t,u_i,\varphi_0)=u_i+\int_{0}^{t}\varphi_0(s)ds\ \
(i=1,2).\nonumber
\end{equation}
By Theorem \ref{thP0} we have
\begin{equation}\label{eqSP2.1}
\mathfrak N_{\varphi_{0}}\subseteq \mathfrak N_{x_i},
\end{equation}
where $x_i=(u_i,\varphi_{0})$ ($i=1,2$). Note that $x_1,x_2\in
M_{\varphi_{0}}\subseteq M$ and taking into account the minimality
of the set $M$ we conclude that there exists a sequence
$\{\bar{t}_{n}\}\subset \mathbb T$ such that
\begin{equation}\label{eqSP3}
\lim\limits_{n\to \infty}\pi(\bar{t}_{n},x_1)=x_{2} .
\end{equation}
Notice that from (\ref{eqSP3}) it follows that
\begin{equation}\label{eqSP0}
\{\bar{t}_{n}\}\in \mathfrak N_{\varphi_{0}}.\nonumber
\end{equation}
Indeed, we have
$$
\varphi_{0}=h(x_2)=h(\lim\limits_{n\to
\infty}\pi(\bar{t}_{n},x_1))=\lim\limits_{n\to
\infty}h(\pi(\bar{t}_{n},x_1))=
$$
\begin{equation}\label{eqSP3.1}
\lim\limits_{n\to
\infty}\sigma(\bar{t}_{n},h(x_1))=\lim\limits_{n\to
\infty}\sigma(\bar{t}_{n},\varphi_{0})
\end{equation}
and, consequently, $\{\bar{t}_{n}\}\in \mathfrak N_{\varphi_{0}}$.

On the other hand we have
\begin{equation}\label{eqSP4}
\rho(\pi(\bar{t}_{n},x_1),\pi(\bar{t}_{n},x_2))\le
\rho(\pi(\bar{t}_{n},x_1),x_2)+\rho(x_2,\pi(\bar{t}_{n},x_2)
\end{equation}
for every $n\in \mathbb N$. Passing to the limit in (\ref{eqSP4}) as
$n\to \infty$ and taking into account (\ref{eqSP2.1}) and
(\ref{eqSP3.1}) we obtain
\begin{equation}\label{eqSP5}
\lim\limits_{n\to
\infty}\rho(\pi(\bar{t}_{n},x_1),\pi(\bar{t}_{n},x_2))=0.
\end{equation}
Note that
\begin{equation}\label{eqSP6}
\rho(x_1,x_2)\le
\rho(x_1,\pi(\bar{t}_{n},x_1))+\rho(\pi(\bar{t}_{n},x_1),\pi(\bar{t}_{n},x_2))+\rho(x_2,\pi(\bar{t}_{n},x_2))
\end{equation}
for all $n\in \mathbb N$. Taking into account (\ref{eqSP2.1}) and
(\ref{eqSP5}) and passing to the limit in (\ref{eqSP6}) as $n\to
\infty$ we obtain $x_1=x_2$ and, consequently, $u_1=u_2$
($x_i=(u_i,\varphi_{0}),\ i=1,2$). The last equality contradicts
to our assumption. The obtained contradiction proves our statement
that the set $M_{y}$ consists of a singe point for every $y\in Y$.

Let now $x_{\varphi}=(u_{\varphi},\varphi)\in M_{\varphi}$. We
will show that the primitive
\begin{equation}\label{eqSP6.1}
F(t):=\phi(t,u_{\varphi},\varphi)=u_{\varphi}+\int_{0}^{t}\varphi(s)ds
\nonumber
\end{equation}
of the function $\varphi$ is strongly comparable by the character
of recurrence with the function $\varphi$, i.e., $\mathfrak
M_{\varphi}\subseteq \mathfrak M_{F}$ or equivalently $\mathfrak
M_{\varphi}\subseteq \mathfrak M_{x_{\varphi}}$. Indeed, if
$\{t_n\}\in \mathfrak M_{\varphi}$, then $\{t_n\}\in \mathfrak
M_{x_{\varphi}}$. Assuming that it is not so then there exists a
sequence $\{t_n\}\in \mathfrak M_{\varphi}$ such that
$\{t_n\}\notin \mathfrak M_{x_{\varphi}}$. This means that the
sequence $\{\pi(t_n,x_{\varphi})\}$ is not convergent. Since
$\{\pi(t_n,x_{\varphi})\}\subseteq M$ and the set $M$ is compact
then the sequence $\{\pi(t_n,x_{\varphi})\}$ is precompact and it
has at least two different limit points $\bar{x}_{i}$ ($i=1,2$).
Since $\{t_n\}\in \mathfrak M_{\varphi}$ then there exists a point
$\bar{\varphi}\in H(\varphi)$ such that $\varphi^{t_n}\to
\bar{\varphi}$ as $n\to \infty$. It easy to check that
$\bar{x}_{i}\in M_{\bar{\varphi}}$ ($i=1,2$). According to our
assumption $\bar{x}_{1}\not= \bar{x}_{2}$. Thus we have a point
$\varphi_{0}\in H(\varphi)$ for which the set $M_{\bar{\varphi}}$
contains at least two different points, but this fact contradicts
what has been proven above. The obtained contradiction complete
the proof of Theorem.
\end{proof}

\begin{coro}\label{corP2.01} Let $\psi \in L^{p}_{loc}(\mathbb T,\mathfrak
B)$ be a $S^{p}$ almost periodic (respectively, $S^{p}$
recurrent). Every precompact primitive of the function $\psi$ is
almost periodic (respectively, recurrent).
\end{coro}
\begin{proof} This statement follows from Theorems \ref{thP1}
and \ref{thRAP4.1}.
\end{proof}

\begin{coro}\label{corP2.1} Under the condition of Theorem
\ref{thP1} for every precompact primitive $F$ of the function
$\varphi$ there exists a continuous function $\nu
:\omega_{\varphi}\to \mathfrak B$ such that
\begin{equation}\label{eqN_1}
\phi(t,\nu(\psi),\psi)=\nu(\sigma(t,\psi))\nonumber
\end{equation}
for all $(t,\psi)\in \mathbb R\times \omega_{\varphi}$.
\end{coro}
\begin{proof} Let $F$ be a compact primitive of the function
$\varphi$ then there exists an element $u\in \mathfrak B$ such
that
\begin{equation}\label{eqS1}
F(t)=\phi(t,u,\varphi)=u+\int_{0}^{t}\varphi(s)ds \nonumber
\end{equation}
for all $t\in \mathbb T$. By Theorem \ref{thP1} $\mathfrak
M_{\varphi}\subseteq \mathfrak M_{x},$ where $x=(u,\varphi)\in
X:=\mathfrak B\times H(\varphi)$. According to Theorem
\ref{thRAP4.01} there exists a continuous mapping
$\gamma:H(\varphi)\to H(x)$ (or equivalently, there exists a
continuous mapping $\nu :\omega_{\varphi}\to \mathfrak B$)
satisfying the conditions
\begin{equation}\label{eqS2}
\gamma(\varphi)=x=(u,\varphi)\ \ \mbox{and}\ \
\gamma(\sigma(t,\psi))=\pi(t,\gamma(\psi)) \nonumber
\end{equation}
(or equivalently
\begin{equation}\label{eqS3}
\nu(\varphi)=u\ \ \mbox{and}\ \ \
\nu(\sigma(t,\psi))=\phi(t,\nu(\psi),\psi) ) \nonumber
\end{equation}
for all $(t,\psi)\in \mathbb T\times \omega_{\varphi}$.
\end{proof}

Below we will consider a special class of remotely almost periodic
functions. Namely we consider the remotely almost periodic
functions $\varphi$ with the minimal $\omega$-limit set
$\omega_{\varphi}$, i.e., $\omega_{\varphi}$ is a minimal set of
shift dynamical system $(L^{p}_{loc}(\mathbb T,\mathfrak
B),\mathbb T,\sigma)$.

\begin{theorem}\label{thIC2} Let $\varphi \in L^{p}_{loc}(\mathbb T,\mathfrak
B)$ be a positively Lagrange stable function. Assume that the
following conditions are fulfilled:
\begin{enumerate}
\item $\varphi$ is a $S^{p}$ remotely almost periodic
(respectively, remotely $\tau$-periodic or remotely stationary)
function and its $\omega$-limit set $\omega_{\varphi}$ is minimal;
\item the primitive $F(t):=\int_{0}^{t}\varphi(s)ds$ of the
function $\varphi$ is precompact, i.e.,
$Q:=\overline{\varphi(\mathbb R_{+})}$ is a compact subset of
$\mathfrak B$.
\end{enumerate}

Then the primitive $F$ of $\varphi$ is positively remotely almost
periodic (respectively, positively remotely $\tau$-periodic or
positively remotely stationary) function.
\end{theorem}
\begin{proof} Denote by $Y:=H(\varphi)$ the closure of the set
$\{\sigma(t,\varphi)|\ t\in \mathbb T\}$ in the space
$L^{p}_{loc}(\mathbb T,\mathfrak B)$ and $(Y,\mathbb T,\sigma)$
the shift dynamical system on $Y$ induced by $(L^{p}_{loc}(\mathbb
T,\mathfrak B),$ $\mathbb T,$ $\sigma)$, $X:=\mathfrak B\times Y$
and $(X,\mathbb T,\pi)$ the skew-product dynamical system:
$\pi(t,(v,\psi)):=(\phi(t,v,\psi),\sigma(t,\psi))$ for all
$(t,v,\psi)\in \mathbb T\times \mathfrak B\times Y$, where
\begin{equation}\label{eqG2}
\phi(t,v,\psi):=v+\int_{0}^{t}\psi(s)ds .\nonumber
\end{equation}
Note that the mapping $\phi :\mathbb T\times \mathfrak B\times
Y\to \mathfrak B$ possesses the following properties:
\begin{enumerate}
\item $\phi(0,v,\psi)=v$ for every $v\in \mathfrak B$ and $\psi\in
Y$; \item
$\phi(t+\tau,v,\psi)=\phi(t,\phi(\tau,v,\psi),\sigma(\tau,\psi))$
for all $t,\tau\in \mathbb T$, $v\in \mathfrak B$ and $\psi \in
Y$; \item the mapping $\phi$ is continuous.
\end{enumerate}
The first two properties are evident. The third property follows
from Lemma \ref{lISP1} (item (iii)).

Thus $\langle \mathfrak B,\phi, (Y,\mathbb T,\sigma)\rangle$
(shortly $\phi$) is a cocycle over dynamical system $(Y,\mathbb
T,\sigma)$ with the fibre $\mathfrak B$ associated by $\phi$.

Consider a nonautonomous dynamical system $\langle (X,\mathbb
T,\pi), (Y,\mathbb T,\sigma),h\rangle$, where $h:X\to Y$ is
defined by $h:= pr_{2}$ generated by the cocycle $\phi$.

Since the function $\varphi$ is $S^{p}$ positively Lagrange stable
and the primite $F$ of $\varphi$ is precompact then by Corollary
\ref{corISP2} it is also positively Lagrange stable in the shift
dynamical system $(C(\mathbb T,\mathfrak B),\mathbb T,\sigma)$.

Denote by $x_0:=(0,\varphi)\in X$, where $0$ is the null element
of $\mathfrak B$. Under the condition of Theorem the point $x_0\in
X$ is positively Lagrange stable in $(X,\mathbb T,\pi)$. Note that
$\widetilde{X}:=H^{+}(x_0)$ is a nonempty, compact and positively
invariant subset of $(X,\mathbb T,\pi)$. Since $\varphi =h(x_0)$
is positively $S^{p}$ Lagrange stable then
$h(H^{+}(x_0))=H^{+}(\varphi)$ and, consequently, the
nonautonomous dynamical system $\langle (X,\mathbb
T,\pi),(Y,\mathbb T,\sigma),h\rangle$ induces on
$\widetilde{X}=H^{+}(x_0)$ (respectively, on $\omega_{x_0}$) a
semi-group (respectively, a two-sided) nonautonomous dynamical
system $\langle (\widetilde{X},\mathbb
R_{+},\pi),(\widetilde{Y},\mathbb R_{+},\sigma),h\rangle$
(respectively, $\langle (\omega_{x_0},$ $\mathbb R,$ $\pi),$
$(\omega_{\varphi},$ $\mathbb R,$ $\sigma),$ $h\rangle$), where
$\widetilde{Y}:=H^{+}(\varphi)$.

Now we will show that there exist at least one continuous mapping
$\nu :\omega_{\varphi}\to \mathfrak B$ such that
\begin{equation}\label{eqN1}
\nu(\sigma(t,\psi))=\phi(t,\nu(\psi),\psi) \nonumber
\end{equation}
for all $(t,\psi)\in \mathbb R\times \omega_{\varphi}$. To this
end we fix a point $\widetilde{x} \in \omega_{x_0}.$

Since the point $x_0$ is positively Lagrange stable in the
skew-product dynamical system $(X,\mathbb T,\pi)$ then
$\omega_{x_0}$ is a nonempty, compact and invariant set. For the
point $\widetilde{x}\in \omega_{x_0}$ there exists a sequence
$t_n\to +\infty$ such that
$\pi(t_n,x_0)=(\phi(t_n,x_0,\varphi),\sigma(t_n,\varphi))\to
(\widetilde{u},\widetilde{\psi})=\widetilde{x}$ as $n\to \infty$.

Note that
\begin{equation}\label{eqMRAP1}
F^{t_n}(t)=F^{t_n}(0)+\int_{0}^{t}\varphi^{t_n}(s)ds
\end{equation}
for all $t\in \mathbb T$ and, consequently, passing to the limit
in (\ref{eqMRAP1}) as $n\to \infty$ we obtain
\begin{equation}\label{eqMRAP1.1}
\widetilde{F}(t)=\widetilde{F}(0)+\int_{0}^{t}\widetilde{\psi}(s)ds
.\nonumber
\end{equation}
Taking into account that $\widetilde{F}(\mathbb R)\subseteq
Q=\overline{\varphi(\mathbb R_{+})}$ we conclude that the function
$\widetilde{F}$ is a compact primitive of the function
$\widetilde{\psi}\in \omega_{\psi}$. On the other hand since the
function $\varphi$ is $S^{p}$ remotely almost periodic then by
Theorem \ref{th1SRAP} the function $\widetilde{\psi}$ is $S^{p}$
almost periodic and by Corollary \ref{corP2.1} there exist a
continuous mapping $\nu :\omega_{\varphi}\to \mathfrak B$ such
that
\begin{equation}\label{eqMRAP1.2}
\nu(\widetilde{\psi})=\widetilde{u}\ \ \mbox{and}\ \
\nu(\sigma(t,\psi))=\phi(t,\nu(\psi),\psi) \nonumber
\end{equation}
for all $(t,\psi)\in \mathbb R\times \omega_{\varphi}$.

Now we will show that the $\omega$-limit set $\omega_{x_0}$ of
$x_0$ (or equivalently, the $\omega$-limit set
$\omega_{(F,\varphi)}$ of $(F,\varphi)$) is an equi-almost
periodic set. To prove this fact we fix an arbitrary positive
number $\varepsilon$.

By uniform continuity of the map $\gamma : \omega_{\varphi}\to
\omega_{x_0}$ (respectively, the map $\nu :\omega_{\varphi}\to
I:=pr_1(\omega_{x_0})$ because $\gamma =(\nu,Id_{Y})$) for
$\varepsilon/2$ there exists a positive number
$\delta=\delta(\varepsilon/2)$
($0<\delta(\varepsilon)<\varepsilon$) such that
$\rho_{S^{p}}(\theta_{1},\theta_{2})<\delta$ implies
\begin{equation}\label{eqMRAP2}
\rho_{X}(\gamma(\theta_{1}),\gamma(\theta_{2}))<\varepsilon/2\ \
(\mbox{respectively},\ \rho_{\mathfrak B}(\nu
(\theta_1),\nu(\theta_2))<\varepsilon/2) ,
\end{equation}
where $\rho_{\mathfrak B}(u,v):=|u-v|$ for all $u,v\in \mathfrak
B$, $\rho_{X}(x_{1},x_{2})=\rho_{\mathfrak
B}(v_1,v_2)+\rho_{S^{p}}(\theta_{1},\theta_{2})$ and
$x_{i}=(v_{i},\theta_{i})$ (i=1,2).

Since the function $\varphi$ is positively Lagrange stable and
remotely almost periodic then by Theorem \ref{th1SRAP} the set
$\omega_{\varphi}$ is $S^{p}$ equi-almost periodic. This means
that the set
\begin{equation}\label{eqMRAP3}
\mathcal F(\varepsilon/2,\omega_{\varphi}):=\{\tau \in \mathbb R|\
\sup\limits_{t\in \mathbb T,\ \theta \in
\omega_{\varphi}}\rho_{S^{p}}(\sigma(t+\tau,\theta),\sigma(t,\theta))<\delta(\varepsilon/2)\}
\end{equation}
is relatively dense in $\mathbb R$.

If $\tau\in \mathcal F(\varepsilon/2,\omega_{\varphi})$, then from
(\ref{eqMRAP2}) and (\ref{eqMRAP3}) we have
$$
\rho_{X}(\phi(t+\tau,\gamma(\theta),\theta),\phi(t,\gamma(\theta),\theta))=
\rho_{X}(\gamma(\sigma(t+\tau,\theta)),\gamma(\sigma(t,\theta))<\varepsilon/2
$$
(respectively,
$$
\rho_{\mathfrak
B}(\phi(t+\tau,\nu(\theta),\theta),\phi(t,\nu(\theta),\theta))=
\rho_{\mathfrak
B}(\nu(\sigma(t+\tau,\theta)),\nu(\sigma(t,\theta))<\varepsilon/2
)
$$
for all $t\in \mathbb R$ and $\theta \in \omega_{\varphi}$, i.e.,
$\tau \in \mathcal F(\varepsilon, \omega_{\varphi})$.

Let now $(\Phi,\psi)\in \omega_{(F,\varphi)}$ be an arbitrary
point then as we established above $\Phi'(t)=\psi(t)$ for all
$t\in \mathbb R$. On the other hand
$\phi(t,\nu(\psi),\psi)'=\psi(t)$ and, consequently,
\begin{equation}\label{eqMRAP6}
\Phi (t)=\nu(\sigma(t,\psi))+\Phi(0)-\nu(\psi)
\end{equation}
for all $t\in \mathbb R$. From (\ref{eqMRAP6}) we obtain
\begin{equation}\label{eqMRAP07}
\rho_{\mathfrak B}(\Phi(t+\tau),\Phi(t))=\rho_{\mathfrak
B}(\nu(\sigma(t+\tau,\psi),\gamma(\sigma(t,\psi)))<\varepsilon/2 .
\end{equation}

According to (\ref{eqMRAP3}) and (\ref{eqMRAP07}) we obtain
\begin{equation}\label{eqMRAP7}
\rho_{X}((\Phi(t+\tau),\sigma(t+\tau,\psi)),\Phi(t),\sigma(t,\psi))=\nonumber
\end{equation}
$$
\rho_{\mathfrak
B}(\Phi(t+\tau),\Phi(t))+\rho_{S^{p}}(\sigma(t+\tau,\psi),\sigma(t,\psi))<\varepsilon/2+\varepsilon/2=\varepsilon
$$
for all $t\in \mathbb T$ and $(\Phi,\psi)\in
\omega_{(F,\varphi)}$.

This means that the set $\omega_{(F,\varphi)}$ is equi-almost
periodic. To finish the proof we note that if the set
$\omega_{\varphi}$ consists of $\tau$-periodic  (respectively,
stationary) points, then the set $\omega_{(F,\varphi)}$ is so.
Theorem is proved.
\end{proof}

\begin{coro}\label{corIC1} Let $\varphi \in C(\mathbb R,\mathfrak
B)$ be a positively Lagrange stable function. Assume that the
following conditions are fulfilled:
\begin{enumerate}
\item $\varphi$ is a positively remotely almost periodic
(respectively, positively remotely $\tau$-periodic or positively
remotely stationary) function and its $\omega$-limit set
$\omega_{\varphi}$ is minimal; \item the primitive
$F(t):=\int_{0}^{t}\varphi(s)ds$ of the function $\varphi$ is
precompact.
\end{enumerate}

Then the primitive $\Phi$ of $\varphi$ is positively remotely
almost periodic (respectively, positively remotely $\tau$-periodic
or positively remotely stationary) function.
\end{coro}
\begin{proof} This statement directly follows from Theorem \ref{thIC2}
because every remotely almost periodic (respectively, remotely
$\tau$-periodic or remotely stationary) function $\varphi \in
C(\mathbb R,\mathfrak B)$ is $S^{p}$ remotely almost periodic
(respectively, remotely $\tau$-periodic or remotely stationary).
\end{proof}

\begin{remark}\label{remIC1} Corollary \ref{corIC1} confirm the conjecture
formulated in the work \cite{Che_2024.2}.
\end{remark}

\begin{remark}\label{remIC2} Note that under the conditions of
Corollary \ref{corIC1} the compact primitive
$F(t)=\int_{0}^{t}\varphi(s)ds$ of the function $\varphi \in
C(\mathbb T,\mathfrak B)$ is remotely stationary (respectively,
remotely $\tau$-stationary or remotely almost periodic). But the
$\omega$-limit set $\omega_{F}$ of $F$ can be either minimal or
non-minimal.
\end{remark}

\section{Examples}\label{Sec4}

Below we will give two examples illustrating Remark \ref{remIC2}.
Namely an example with minimal $\omega_{F}$ (Example \ref{exIC1})
an other one with the non minimal set $\omega_{F}$ (Example
\ref{exIC2}).

\begin{example}\label{exIC1}
{\em We note (see, for example, \cite{Che_2024.2}) that the
function $\psi (t)=\sin(t+\ln(1+|t|))$ ($t\in \mathbb R$)
possesses the following properties:
\begin{enumerate}
\item $\psi$ is remotely $2\pi$-periodic; \item the set
$\omega_{\psi}$ is minimal; \item $\psi$ is not asymptotically
almost periodic.
\end{enumerate}

Consider the function $\varphi\in C(\mathbb R_{+},\mathbb R)$
defined by the equality $\varphi(t):=\cos (t+\ln (1+t))$ for every
$t\in \mathbb R_{+}$. The function $\varphi$ possesses the
following properties:
\begin{enumerate}
 \item
$$
|\cos(t+2\pi +\ln(1+t+2\pi))-\cos(t+\ln(1+t))|=
$$
$$
2|\sin(t+\frac{\ln(1+t+2\pi)(1+t)}{2})\sin(\frac{\ln(1+\frac{2\pi}{1+t})}{2})|\le
$$
$$
|\ln(1+\frac{2\pi}{1+t})|\to 0
$$
as $t\to +\infty$ and, consequently, the function $\varphi$ is
remotely $2\pi$-periodic; \item $\varphi$ is bounded on the
$\mathbb R_{+}$; \item
\begin{equation}\label{eqIC1.03}
\varphi'(t)=-(1+\frac{1}{1+t})\sin(t+\ln(1+t)) \nonumber
\end{equation}
for all $t\in \mathbb R_{+}$. It is easy to see that the
derivative $\varphi'$ is bounded on $\mathbb R_{+}$ and,
consequently, the function $\varphi$ is uniformly continuous on
$\mathbb R_{+}$; \item the function $\varphi$ is positively
Lagrange stable and, consequently, its $\omega$-limit set
$\omega_{\varphi}$ is nonempty, compact and invariant (in the
shift dynamical system $(C(\mathbb R_{+},\mathbb R),\mathbb
R_{+},\sigma)$); \item
\begin{equation}\label{eqIC1.04}
\omega_{\varphi}=\{\cos^{h}:\ h\in [0,2\pi)\}.\nonumber
\end{equation}
To prove this equality we note that the function
$\widetilde{\varphi}\in \omega_{\varphi}$ if and only if there exists
a sequence $\{h_k\}\subset \mathbb R_{+}$ such that $h_k\to
+\infty$ and $\varphi^{h_k}\to \widetilde{\varphi}$ (in the space
$C(\mathbb R_{+},\mathbb R)$) as $k\to \infty$.

It easy to see that
$$
\widetilde{\varphi}(t)=\lim\limits_{k\to
\infty}\cos(r+h_{k}+\ln(1+t+h_k))=
$$
\begin{equation}\label{eqF1}
\lim\limits_{k\to \infty}\cos(t+h_{k} +\ln (2\pi
h_{k})+\ln(1+\frac{1+t}{h_k})
\end{equation}
for every $k\in \mathbb N$.

Note that
\begin{equation}\label{eqF2}
 h_k +\ln (2\pi m_{k})=2\pi m_k+\tau_{k}
 \end{equation}
 for all $k\in
\mathbb N$, where $m_k\in \mathbb N$. Without loos of the
generality we can suppose that the sequences $\{\tau_{k}\}$ is
convergent and its limit $\widetilde{\tau}\in [0,2\pi)$, i.e.,
\begin{equation}\label{eqF3}
\widetilde{\tau}=\lim\limits_{k\to \infty}\tau_{k} .
\end{equation}
Passing to the limit in the equality (\ref{eqF1}) and taking into
account (\ref{eqF2})-(\ref{eqF3}) we have
\begin{equation}\label{eqF4}
\widetilde{\varphi}(t)=\cos(t+\tau_{0}) \nonumber
\end{equation}
for every fixed $t\in \mathbb R_{+}$, that is,
$\widetilde{\varphi}=\cos^{\tau_{0}}$ (for some $\tau_{0}\in
[0,2\pi)$). Thus the equality is established. In particular we
obtain that the set $\omega_{\varphi}$ is minimal. \item the
function $\varphi$ is not asymptotically $2\pi$-periodic.

If we assume that it is not true, then there exists a
$2\pi$-periodic function $\widetilde{\varphi}\in \omega_{\varphi}$
(and, consequently, there exists a number $\alpha \in [0,2\pi)$
with property $\widetilde{\varphi}=\cos^{\alpha}$) such that
\begin{equation}\label{eqF5}
\lim\limits_{t\to +\infty}|\varphi(t)-\widetilde{\varphi}(t)|=0 .
\end{equation}

Note that
\begin{equation}\label{eqF6}
\psi^{2}(t)+\varphi^{2}(t)=\sin^{2}(t+\ln
(1+t))+\cos^{2}(t+\ln(1+t))=1 \nonumber
\end{equation}
and, consequently,
\begin{equation}\label{eqF7} .
\psi^{2}(t)=1-\varphi^{2}(t)
\end{equation}
for all $t\in \mathbb R_{+}$. From (\ref{eqF7}), taking into
account (\ref{eqF5}), we conclude that the function $\psi$ is
asymptotically $2\pi$-periodic. The last statement contradicts to
the choice of $\psi$. The obtained contradiction prove our
statement. \item
\begin{equation}\label{eqF8}
\psi'(t)=(1+\frac{1}{1+t})\varphi(t):=\mu(t) \nonumber
\end{equation}
for all $t\in \mathbb R_{+}$, i.e., the function $\psi$ is a
bounded primitive of the remotely $2\pi$-periodic function $\mu$
(as the product of two remotely $2\pi$-periodic functions) with
minimal $\omega$-limit set $\omega_{\mu}=\omega_{\varphi}$.
\end{enumerate}
}
\end{example}

\begin{example}\label{exIC2} Consider the function $\varphi \in C(\mathbb R_{+},\mathbb
R)$defined by the equality
\begin{equation}\label{eqF9}
\varphi(t)=\frac{2t}{(\pi^{3}+t^{2})^{2/3}}\cos
(\pi^{3}+t^{2})^{1/3}\nonumber
\end{equation}
for every $t\in \mathbb R_{+}$.

The function $\varphi$ possesses the following properties:
\begin{enumerate}
\item $\lim\limits_{t\to +\infty}|\varphi(t)|=0$ and,
consequently, the function $\varphi$ is asymptotically stationary
and $\omega_{\varphi}=\{0\}$. Thus its $\omega$-limit set is
minimal; \item $F(t)=\int_{0}^{t}\varphi(s)ds=\sin
(\pi^{3}+t^2)^{1/3}$ for all $t\in \mathbb R_{+}$; \item for every
fixed $\tau
>0$ we have
$$
|F(t+\tau)-F(t)|=
|\sin(\pi^{3}+(t+\tau)^{2})^{1/3}-\sin(\pi^{3}+t^{2})^{1/3}|=
$$
$$
2|\sin\frac{(\pi^{3}+(t+\tau)^{2})^{1/3}-(\pi^{3}+t^{2})^{1/3}}{2}
\cos(\frac{(\pi^{3}+(t+\tau)^{2})^{1/3}+(\pi^{3}+t^{2})1^{1/3}}{2})|=
$$
$$
2|\sin\frac{\tau (t+\tau)}{2((\pi^{3}+(t+\tau)^{2})^{2/3}+
(\pi^{3}+(t+\tau)^{2})^{1/3}(\pi^{3}+t^{2})^{1/3}
+(\pi^{3}+t^{2})^{2/3})}|\cdot
$$
$$
|\cos(\frac{(\pi^{3}+(t+\tau)^{2})^{1/3}+(\pi^{3}+t^{2})^{1/3}}{2})|\le
$$
$$
\frac{\tau (t+\tau)}{(\pi^{3}+(t+\tau)^{2})^{2/3}+
(\pi^{3}+(t+\tau)^{2})^{1/3}(\pi^{3}+t^{2})^{1/3}
+(\pi^{3}+t^{2})^{2/3}} \to 0
$$
as $t\to +\infty$, i.e., the function $F$ is remotely stationary;
\item the primitive $F$ is bounded and uniformly continuous on
$\mathbb R_{+}$ and, consequently, $F$ is positively Lagrange
stable (thus its $\omega$-limit set is nonempty, compact and
invariant); \item $\omega_{F}=\{F_{c}:\ c\in [-1,1]\}$, where
$F_{c}(t)=c$ for all $t\in \mathbb R$ (for the details see
\cite{Che_2024.2}).
\end{enumerate}
\end{example}

\textbf{Problem.} Assume that the Banach space $\mathfrak B$ does
not contain a subspace isomorphic to $c_{0}$. Find out which
results of Section \ref{Sec3} are preserved when replacing the
condition of "precompactness" with "boundedness".

\textbf{Open problem.} The question, is whether Theorem
\ref{thIC2} remains true in the general case when the set
$\omega_{\varphi}$ is not minimal, is open.

\section{Funding}

This research was supported by the State Program of the Republic
of Moldova "Remotely Almost Periodic Solutions of Differential
Equations (25.80012.5007.77SE)" and partially was supported by the
Institutional Research Program 011303 "SATGED", Moldova State
University.

\section{Data availability}

No data was used for the research described in the article.

\section{Conflict of Interest}

The author declares that he does not have conflict of interest.

\end{document}